\documentclass[reqno]{amsart}
\usepackage{amsmath}
\usepackage{amsthm}
\usepackage{amsfonts}
\usepackage{amssymb}
\usepackage{graphicx}
\usepackage{mathrsfs}
\usepackage{cite}
\usepackage{textcomp}

\title[POINT DISTRIBUTIONS IN TWO-POINT HOMOGENEOUS SPACES]{POINT DISTRIBUTIONS IN COMPACT METRIC SPACES, III. TWO-POINT HOMOGENEOUS SPACES}

\author{M.M. SKRIGANOV}

\address{St. Petersburg Department, Steklov Mathematical Institute, Russian Academy of Sciences}

\email{maksim88138813@mail.ru}



\numberwithin{equation}{section}

\newtheorem{theorem}{Theorem}[section]
\newtheorem{lemma}{Lemma}[section]
\theoremstyle{remark}

\theoremstyle{remark}
\newtheorem{definition}{Definition}[section]

\theoremstyle{plain}
\newtheorem{corollary}{Corollary}[section]

\DeclareMathOperator{\diam}{\mathrm{diam}}

\DeclareMathOperator{\Lip}{\mathrm{Lip}}
\DeclareMathOperator{\Spin}{\mathrm{Spin}}
\DeclareMathOperator{\RE}{\mathrm{Re}}
\DeclareMathOperator{\Tr}{\mathrm{Tr}}
\DeclareMathOperator{\Diam}{\mathrm{Diam}}

\def\M{\mathcal M}

\def\Q{\mathbb Q}
\def\F{\mathbb F}
\def\bR{\mathbb R}
\def\la{\lambda}

\def\lan{\langle}
\def\ran{\rangle}
\def\D{\mathcal D}
\def\R{\mathcal R}
\def\E{\mathcal E}
\def\O{\mathcal O}
\def\HH{\mathcal H}

\def\C{\mathbb C}
\def\H{\mathbb H}
\def\bO{\mathbb O}
\def\bR{\mathbb R}

\def\bE{\mathbb E}
\def\bZ{\mathbb Z}

\numberwithin{equation}{section}
\theoremstyle{plain}

\newcommand{\bp}{\begin{proof}}
\newcommand{\ep}{\end{proof}}
\newcommand{\bl}{\begin{lemma}}
\newcommand{\el}{\end{lemma}}
\newcommand{\bt}{\begin{theorem}}
\newcommand{\et}{\end{theorem}}
\newcommand{\bd}{\begin{definition}}
\newcommand{\ed}{\end{definition}}
\newcommand{\ba}{\begin{arrow}}
\newcommand{\ea}{\end{arrow}}

\begin{document}


%
%
%
%




\begin{abstract}

We consider point distributions in compact connected
two-point homogeneous spaces (Riemannian symmetric spaces of rank one).
All such spaces are known, they are the spheres in the Euclidean spaces, 
the real, complex and quaternionic projective spaces and the octonionic projective plane.
Our concern is with discrepancies of distributions in metric balls and sums of 
pairwise distances between points of distributions in such spaces. 

Using the geometric features of two-point spaces, we show that Stolarsky's invariance
principle, well-known for the Euclidean spheres, can be extended to all 
projective spaces and the octonionic projective plane (Theorem~\ref{thm2.1} 
and Corollary~\ref{cor2.1}). 
We obtain the spherical function expansions for discrepancies and sums of 
distances (Theorem~\ref{thm9.1}). Relying on these expansions, we prove in all such 
spaces the best possible bounds for quadratic discrepancies and sums of pairwise distances (Theorem 2.2). 
Applications to $t$-designs on such two-point homogeneous
spaces are also considered. It is shown that the optimal $t$-designs,
recently constructed in \cite{9, 10, 17a}, meet the best possible bounds for quadratic discrepancies and sums of pairwise distances. 
(Corollaries~\ref{cor3.1}, \ref{cor3.2}).  

\end{abstract}



\maketitle

\thispagestyle{empty}



\section*{Contents}

\noindent{} {\bf I. Main results} 

\noindent{} 1. Discrepancies and metrics

\noindent{} 2. Statements of the main results

\noindent{} 3. Applications to $t$-designs 


\noindent{} {\bf II. Geometry of two-point homogeneous spaces \\ and $L_2$-invariance principles}

\noindent{} 4. Preliminaries: Models of projective spaces and chordal metrics 

\noindent{} 5. Proof of Theorem 2.1

\noindent{} 6. Proof of Lemma 2.1

\noindent{} {\bf III. Spherical functions and bounds for discrepancies \\and sums of distances} 

\noindent{}  7. Preliminaries: Commutative spaces and spherical functions

\noindent{}  8. Spherical function expansions for discrepancies and metrics

\noindent{} 9. Bounds for Fourier-Jacobi coefficients

\noindent{} 10. Proof of Theorems 2.2 and 3.1

\noindent{} 11. Additional remarks

References
\medskip

\section*{I. Main results}

\section{Discrepancies and metrics}\label{sec1}

In this section we introduce the basic notation and recall necessary facts
from our previous paper \cite{30} on relationships between discrepancies 
and metrics on general compact metric spaces.

Let $\M$ be a compact connected metric space with a fixed metric $\theta$ 
and a finite Borel measure $\mu$, normalized by
\begin{equation}
\diam (\M,\theta)=\pi, \quad \mu (\M)=1,
\label{eq1.1}
\end{equation}
where 
$\diam (\E,\rho)=\sup \{ \rho(x_1,x_2): x_1,x_2\in \E\}$
denotes the diameter of a subset $\E\subseteq \M$  with respect to a metric $\rho$. 

We write $B_r(y)=\{x:\theta (x,y)<r\}$ for the ball of radius $r\in \R$
centered at  $y\in \M$ and of volume $v_r(y)=\mu (B_r(y))$, 
here $\R = \{r=\theta(x_1,x_2): x_1,x_2\in \M\}$ is the set of all possible radii. 
Since the space $\M$ is connected, we have $\R = [0,\pi]$.

Let  $\D_N\subset \M$ be a finite subset consisting of $N$ points (not necessary different).
The \emph{local discrepancy} of the subset $\D_N$ in the ball $B_r(y)$ is defined by
\begin{equation}
\Lambda[B_r(y),\D_N]=\#\{B_r(y)\cap \D_N\}-N v_r(y)
\\
=\sum\limits_{x\in \D_N} \Lambda (B_r(y),x),
\label{eq1.3}
\end{equation}  
where
\begin{equation}
\Lambda(B_r(y),x)=\chi(B_r(y),x)- v_r(y),
\label{eq1.4}
\end{equation}
and $\chi(\E,x)$ denotes the characteristic function of s subset $\E\subset \M$. 

The \emph{quadratic discrepancies} are defined by 
\begin{equation}
\la_r[\D_N]=\int\limits_{\M}\Lambda[B_r(y),\D_N]^2\,d\mu(y)
=\sum\limits_{x_1,x_2\in \D_N} \la_r(x_1,x_2),
\label{eq1.5}
\end{equation}
where 
\begin{equation}
\la_r(y_1,y_2)=\int\limits_{\M}\Lambda  (B_r(y),y_1)\Lambda 
(B_r(y),y_2) \, d\mu (y),
\label{eq1.6}
\end{equation}
and  
\begin{equation}
\la [\eta, \D_N] =\int\limits_\R\la_r [\D_N] \eta  (r) \, dr =
\sum\limits_{x_1,x_2\in \D_N} \la (\eta, x_1,x_2),
\label{eq1.7}
\end{equation}
where 
\begin{equation}
\la(\eta, y_1,y_2)=\int\limits_\R\la_r(y_1,y_2)\eta (r) \, dr,
\label{eq1.8}
\end{equation}
here $\eta (r)$, $r\in [0,\pi]$, is a non-negative weight function.
The quantities $\la_r[\D_N]^{1/2}$ and $\la[\eta, \D_N]^{1/2}$ are known as 
$L_2$-discrepancies. In the present paper it is more convenient to deal with 
the quadratic discrepancies \eqref{eq1.5} and \eqref{eq1.7}.

We introduce the following extremal characteristic  
\begin{equation}
\la_N(\eta) =\inf\limits_{\D_N} \la [\eta, \D_N],
\label{eq1.9}
\end{equation}
where the infimum is taken over all $N$-point subsets $\D_N\subset\M$.

In what follows, besides the original metric $\theta$ in the definition
of a compact metric space $\M$,  we shall deal with many different metrics 
on $\M$. For a metric $\rho$ on $\M$ we define 
the \emph{sum of pairwise distances}  
\begin{equation}
\rho [\D_N] =\sum\limits_{x_1,x_2\in \D_N} \rho (x_1,x_2),
\label{eq1.10}
\end{equation}
and introduce the following extremal characteristic 
\begin{equation}
\rho_N=\sup\limits_{\D_N}\rho [\D_N],
\label{eq1.11}
\end{equation}
where the supremum is taken over all $N$-point subsets $\D_N\subset \M$.
We also write $\langle \rho \rangle$ for the average value of a metric  $\rho$,       
\begin{equation}
\langle \rho \rangle =\iint\limits_{\M\times\M} \rho (y_1,y_2) \,
d\mu (y_1) \, d\mu (y_2).
\label{eq1.12}
\end{equation}

The study of the characteristics \eqref{eq1.9} and \eqref{eq1.11} falls within the subjects of the
discrepancy theory and geometry of distances. An extensive literature is devoted to
point distributions on spheres in the Euclidean space. Detailed surveys of the aria 
can be found in  \cite{2, 6, 10a, 12, 29}. 

It was shown in our recent paper \cite{30} that 
nontrivial results on the quantities \eqref{eq1.9} and \eqref{eq1.11} can be established for 
very general metric spaces. Some of these results are given below in 
Theorems~\ref{thm1.1} and 1.2 in the form adapted for use in the present paper. 

Introduce the following  \emph{symmetric difference metrics } on the space $\M$
\begin{equation}
\theta^{\Delta} (\eta, y_1,y_2)=\int\limits_\R\theta^{\Delta}_r
(y_1,y_2)\eta (r)\, dr,
\label{eq1.13}
\end{equation}
where 
\begin{align}
\theta^{\Delta}_r(y_1,y_2) &=\frac 12\mu (B_r(y_1)\Delta B_r (y_2))\notag
\\
&=\frac 12 \Big(v_r(y_1)+v_r(y_2)-2\mu (B_r(y_1)\cap B_r(y_2))\Big),
\label{eq1.14}
\end{align}
and 
$
B_r(y_1)\Delta B_r(y_2) = B_r(y_1)\cup B_r(y_2) \setminus 
B_2 (y_1)\cap B_r(y_2) 
$
is the symmetric difference of the balls  $B_r(y_1)$ and $B_r(y_2)$. We have
\begin{align}
&\chi(B_r(y_1)\Delta B_r (y_2),y) = 
\notag
\\
&\frac12 \Big( \chi(B_r(y_1),y) + \chi(B_r(y_2),y)-2\chi(B_r(y_1),y)\chi(B_r(y_2),y\Big) 
\notag
\\ 
&=|\chi (B_r(y_1),y)-\chi(B_r(y_2),y)|,
\label{eq1.15}
\end{align}
where we write $\chi(B_r(x),y)$ for the characteristic function of ball $B_r(x)$.
The symmetry of the metric $\theta$ implies the following useful relation   
\begin{equation}
\chi (B_r(y),x)=\chi (B_r(x),y)=\chi (r-\theta (x,y)) 
=\chi_r(\theta (x,y)),
\label{eq1.16}
\end{equation}
where $\chi(z)$, $z\in \bR$ is the characteristic function of the half-axis 
$(0,\infty)$, and $\chi_r(\cdot)$ is the characteristic function of
the interval $[0,r)$, $0\le r\le \pi$. From \eqref{eq1.14} and \eqref{eq1.15}, we obtain 
\begin{align}
&\theta^{\Delta}_r(y_1,y_2)  =\frac 12 \int\limits_{\M} \chi
(B_r(y_1)\Delta B_r (y_2)) \, d\mu (y)   \notag
\\
& =\frac 12  \int\limits_{\M} \Big(\chi(B_r(y_1),y) +\chi 
(B_r(y_2),y) -2\chi (B_r(y_1),y) \chi (B_r(y_2),y) \Big) \, d\mu (y) \notag
\\
& 
=\frac 12 \int\limits_{\M} |\chi (B_r(y_1),y)-\chi
(B_r(y_2),y)| \, d\mu  (y)
\label{eq1.17}
\end{align}
With the help of \eqref{eq1.16}, we obtain the following formulas for the average 
values \eqref{eq1.12} of metrics \eqref{eq1.13} and \eqref{eq1.17}  
\begin{align}
\langle \theta^{\Delta} (\eta)\rangle  = \int\limits_\R
\langle \theta^{\Delta}_r\rangle \eta (r) \, dr,
\label{eq1.18}
\end{align}
\begin{align}
\langle \theta^{\Delta}_r\rangle  = \iint\limits_{\M\times \M}
\theta^{\Delta}_r(y_1,y_2) \, d\mu  (y_1)\, d\mu (y_2) 
= \int\limits_{\M} (v_r(y)-v_r(y)^2)\, d\mu (y)
\label{eq1.19}
\end{align}

The symmetric difference of any two subsets coincides with
the symmetric difference of their complements, see \eqref{eq1.15}. Hence
\begin{align}
\theta^{\Delta}_r(y_1,y_2)&= \frac 12 \mu (B'_r(y_1)\Delta B'_r (y_2))  
\notag
\\
&=\frac 12 \Big(v'_r(y_1) +v'_r(y_2)-2\mu (B'_r(y_1) \cap B'_r(y_2))\Big ),
\label{eq1.20}
\end{align}
where $B'_r(y)=\M\setminus B_r(y)$ is the complement of the ball $B_r(y)$, 
\begin{equation}
v'_r(y)=\mu (B'_r(y)) =1-v_r (y), 
\label{eq1.21}
\end{equation}
and the relation \eqref{eq1.19} takes the form
\begin{equation}
\langle \theta^{\Delta}_r \rangle =
\int\limits_{\M} v_r(y)v'_r(y) \, d\mu (y)
\label{eq1.22}
\end{equation}
In \eqref{eq1.17} the balls $B_r(y)$ can be also replaced by their complements $B'_r(y)$. 

A metric space $\M$ is called \emph{distance-invariant}, if the volume  of any
ball $v_r=v_r(y)$ is independent of $y\in \M$, see \cite[p.~504]{24}.
For such spaces the above formulas for the discrepancies and the symmetric difference metrics 
can be essentially simplified. Substituting \eqref{eq1.14} into \eqref{eq1.6}, we obtain
\begin{align} 
\la_r(y_1,y_2)& =\int\limits_{\M}\chi (B_r(y_1),y) \chi
(B_r(y_2),y) \, d\mu (y) -v^2_r       \notag
\\
& = \mu (B_r(y_1)\cap B_r(y_2) -v^2_r,
\label{eq1.23}
\end{align}
and correspondingly, 
\begin{equation}
\la_r[\D_N]=\sum\limits_{y_1,y_2\in \D_N} \mu
(B_r(y_1)\cap B_r(y_2)) -v^2_rN^2.
\label{eq1.24}
\end{equation}
Similarly, the relations \eqref{eq1.14}, \eqref{eq1.20} and 
\eqref{eq1.19}, \eqref{eq1.22} take the form 
\begin{align}
&\theta^{\Delta}_r(y_1,y_2)  =  v_r-\int\limits_{\M} \chi
(B_r(y_1),y)  \chi (B_r(y_2),y)\, d\mu (y) 
\notag
\\
&= v_r-\mu (B_r(y_1)\cap B_r (y_2))             
= v'_r-\mu (B'_r(y_1)\cap B'_r(y_2)),
\label{eq1.25}
\end{align}
 
\begin{equation}
\langle  \theta^{\Delta}_r\rangle =v_r-v^2_r = v_r v'_r,
\label{eq1.26}
\end{equation}
and  
\begin{equation}
\theta^{\Delta}_r[\D_N] =v_rN^2  - \sum\limits_{y_1,y_2\in \D_N} \mu
(B_r(y_1)\cap B_r(y_2)).
\label{eq1.27}
\end{equation}
Integrating these relations with $\eta (r)$, $r\in [0,\pi]$, one can obtain 
the corresponding formulas for the quantities \eqref{eq1.13} and \eqref{eq1.18}.

The typical examples of distance-invariant spaces are homogeneous spaces $\M=G/K$, 
where $G$ is a compact group, $K\subset G$ is a closed subgroup, while $\theta$ and $\mu$ are 
$G$-invariant metric and measure on $\M$. In this case, the quantities (1.6), 
\eqref{eq1.8} and \eqref{eq1.13}, \eqref{eq1.14} are also $G$-invariant:
\begin{equation}
\left. 
\begin{aligned}
\la_r(gy_1,gy_2) & \!=\la_r(y_1,y_2), \, \,
\la(\eta, gy_1, gy_2)= \la (\eta, y_1,y_2),
\\
\theta^{\Delta}_r(gy_1, gy_2) & \!=\! \theta^{\Delta}_r
(gy_1,gy_2),  \,\, \theta^{\Delta} (\eta,gy_1,gy_2) \!=\!\theta^{\Delta} 
(\eta,y_1,y_2),
\\
\mu (B_r(gy_1) & \cap B_r (gy_2))  \! =\! \mu (B_r(y_1)\cap B_r (y_2)),
\end{aligned}
\right\}
\label{eq1.28}
\end{equation}
for all $y_1,y_2\in G/K$,  $g\in G$.

Comparing the relations \eqref{eq1.23}--\eqref{eq1.27}, we arrive to the following
result. This result and its generalizations were given in \cite[Theorems~2.1,~3.1]{30}.

\begin{theorem}\label{thm1.1}(The $L_1$-invariance principles).
Let a compact connected metric space $\M$ with a metric $\theta$
and a measure $\mu$  be distance-invariant. Then we have
\begin{align}
\la_r(y_1,y_2)+\theta^{\Delta}_r(y_1,y_2) & = \langle 
\theta^{\Delta}_r\rangle, 
\label{eq1.29}
\\
\la (\eta,y_1,y_2)+\theta^{\Delta}(\eta ,y_1,y_2) & = \langle 
\theta^{\Delta} (\eta)\rangle ,
\label{eq1.30}
\\
\la (\eta,\D_N)+\theta^{\Delta}(\eta , \D_N) & = \langle 
\theta^{\Delta} (\eta)\rangle  N^2 ,
\label{eq1.31}
\\
\la_N (\eta)+\theta^{\Delta}_N(\eta) & = \langle 
\theta^{\Delta} (\eta)\rangle N^2.
\label{eq1.32}
\end{align}
Here $r\in \R = [0,\pi]$ and $\D_N\subset \M$ is an arbitrary $N$-point subset. 
The equalities \eqref{eq1.30}, \eqref{eq1.31} and \eqref{eq1.32} hold 
with any weight function $\eta$ such that the integrals  \eqref{eq1.7}, \eqref{eq1.8}
and \eqref{eq1.13}, \eqref{eq1.18} converge.  
\end{theorem}

Obviously, the integrals \eqref{eq1.7}, \eqref{eq1.8} and \eqref{eq1.13}, \eqref{eq1.18}
converge for any wight function $\eta$ summable on the interval $[0,\pi]$.
More general conditions of convergence of these integrals for two-point homogeneous 
spaces will be given in Lemma 2.1 below. Notice that the assumption of connectedness of
the space $\M$ in Theorem 1.1 is of no concern, and the measure $\eta (r) \, dr$
in the definitions \eqref{eq1.8} and \eqref{eq1.13}
can be replaced with a measure $d\xi(r)$ on the set of radii $\R$, see \cite[Theorems.~2.1]{30}

The $L_2$-invariance principle, specific for two-point homogeneous spaces, 
will be given in the next section, see Theorem 2.1 and Corollary 2.1. 
Our terminology of $L_1$- and $L_2$-invariance principles will be explained 
in the comments to Corollary 2.1.

To state a further result from \cite{30} we recall the concept of rectifiable spaces,
see \cite{28}.
A compact metric space $\M$ with a metric $\theta$ and a measure $\mu$ is called $d$-rectifiable 
if there exist a measure $\nu$ on the $d$-dimensional unit cube $I^d=[0,1]^d$ absolutely 
continuous with respect to the $d$-dimensional Lebesgue measure on~$I^d$, a measurable subset 
$\O\subset I^d$, and an injective Lipschitz mapping $f:\O\to\M$, such that
$\mu(\M\setminus f(\O))=0$; and
$\mu(\E)=\nu(f^{-1}(\E\cap f(\O))$ for any $\mu$-measurable subset $\E\subset\M$.
Recall that a map $f:\O\subset\bR^d\to\M$ is Lipschitz if
\begin{equation}
\theta(f(Z_1),f(Z_2))\le c\Vert Z_1-Z_2\Vert,
\quad
Z_1,Z_2\in\O,
\label{eq1.33}
\end{equation}
with a positive constant~$c$, and the smallest such constant is called the Lipschitz 
constant of $f$ and denoted by $\Lip(f)$; in \eqref{eq1.33} $\Vert\cdot\Vert$ denotes 
the Euclidean norm in~$\bR^d$.

Notice that any smooth (or piece-wise smooth) compact $d$-dimensional manifold is $d$-rectifiable if in the local coordinates the metric satisfies \eqref{eq1.33}, and the measure is absolutely continuous with respect to the $d$-dimensional Lebesgue measure.
Particularly, any compact $d$-dimensional Riemannian manifold with the geodesic metric $\theta$ and the Riemannian measure $\mu$ is $d$-rectifiable.
In this case, it is known that the condition \eqref{eq1.33} holds; see ~\cite[Chapter~I, Proposition 9.10]{21}.
On the other hand, the condition on the Riemannian measure is obvious because the metric tensor is continuous. 

The following result was established in \cite[Theorem.4.2]{30}. Notice that the proof of this
result is relying on a probabilistic version of Theorem 1.1, see \cite[Theorem 3.1]{30}.

\begin{theorem}\label{thm1.2}
Suppose that a compact metric space $\M$, with a metric $\theta$ and a measure~$\mu$, is $d$-rectifiable.
Write $C=d2^{d-1}\Lip(f)$, where $\Lip(f)$ is the Lipschitz constant of the map $f$ in the definition of $d$-rectifiability of the space $\M$.
Then the following hold:

{\rm(i)} If a metric $\rho$ on $\M$ satisfies the inequality
\begin{equation}
\rho(x_1,x_2)\le c_0\theta(x_1,x_2)
\label{eq1.34}
\end{equation}
with a constant $c_0>0$, then
\begin{equation}
\rho_N\ge\langle\rho\rangle N^2-c_0CN^{1-\frac 1d}.
\label{eq1.35}
\end{equation}

{\rm(ii)} If the metric $\theta^\Delta(\eta)$ satisfies the inequality
\begin{equation}
\theta^\Delta(\eta,x_1,x_2)\le c_0\theta(x_1,x_2)
\label{eq1.36}
\end{equation}
with a constant $c_0>0$, then
\begin{equation}
\theta^\Delta_N(\eta)\ge\langle\theta^\Delta(\eta)\rangle N^2-c_0CN^{1-\frac 1d}
\label{eq1.37}
\end{equation}
and
\begin{equation}
\lambda_N(\eta)\le c_0CN^{1-\frac 1d}.
\label{eq1.38}
\end{equation}

Particularly, the above statements are true for a compact Riemannian 
manifold with the geodesic distance $\theta$ and the Riemannian measure $\mu$.
\end{theorem}

Under such general assumptions one cannot expect that the bounds \eqref{eq1.37} 
and \eqref{eq1.38} are best possible. 
Consider, for example, the $d$-dimensional unit spheres $S^d=\{x\in\bR^{d+1}:\Vert x\Vert=1\}$ with the geodesic (great circle) metric $\theta$ and the standard Lebesgue measure $\mu$ on $S^d$.
In this case, we have
\begin{equation}
\theta_N=\langle\theta\rangle N^2-\varepsilon_N,
\quad
\langle\theta\rangle=\pi/2 ,
\label{eq1.39}
\end{equation}
where $\varepsilon_N=0$ for even $N$ and $\varepsilon_N\le\pi/2$ for odd~$N$.

The appearance of such anomalously small errors in the formula \eqref{eq1.39}
can be easily explained with the help of invariance principle \eqref{eq1.29}.  
We shall discuss this question in section~\ref{sec12}.


In the present paper we shall show that the bounds \eqref{eq1.37} and \eqref{eq1.38} are best possible
for compact connected two-point spaces and general classes of weight functions $\eta$,
see Theorem 2.2 below.

\section{Statements of the main results}\label{sec2}

First of all we 
recall the definition and 
some necessary  facts on two-point homogeneous spaces,
see \cite{7, 21, 22, 36, 37}. Additional facts on the geometry and harmonic analysis
on such spaces will be given in sections 4 and 7.
Let $G=G(\M)$ be the group of isometries of a metric space $\M$ with 
a metric $\theta$, {\it i.e.} $\theta(gx_1,gx_2)=\theta(x_1,x_2)$ for all 
$x_1$, $x_2\in \M$ and $g\in G$. The space $\M$ is called {\it 
two-point homogeneous}, if for any two pairs of points $x_1$, $x_2$ and 
$y_1$, $y_2$ with $\theta(x_1,x_2)=\theta(y_1,y_2)$ there exists an isometry 
$g\in G$, such that $y_1=gx_1$, $y_2=gx_2$. In this case, the group $G$ is obviously
transitive on $\M$ and $\M=G/K$ is a homogeneous space, where
the subgroup $K\subset G$ is the stabilizer of a point $x_0\in \M$. 
Furthermore, the homogeneous space $\M$ is symmetric, {\it i.e.} for any two points 
$y_1$, $y_2\in \M$ there exists an isometry $g\in G$, such that $gy_1=y_2$, $gy_2=y_1$.

We consider compact connected two-point homogeneous spaces $\M=G/K$. For such spaces
$G$ and $K\subset G$ are Lie groups and $\M=G/K$ are Riemannian symmetric spaces of rank one. This means that all flat totally geodesic submanifolds in $\M$ are 
one dimensional and coincide with geodesics. This also means that all $G$-invariant
differential operators on $\M$ are polynomials of the Laplace-Beltrami operator
on $\M$.
All such spaces are classified completely, see, for example,~\cite[Sec.~8.12]{36}. They are the following:

(i) The $d$-dimensional Euclidean spheres 
$S^d=SO(d+1)/SO(d)\times 
\{1\}$, $d\ge 2$, and $S^1=O(2)/O(1) \times \{ 1\}$. 

(ii) The real projective spaces $\bR P^n=O(n+1)/O(n)\times O(1)$.

(iii) The complex projective spaces $\C P^n=U(n+1)/U(n)\times U(1)$.

(iv) The quaternionic projective spaces $\H P^n=Sp(n+1)/SP(n)\times Sp(1)$,

(v) The octonionic projective plane $\bO P^2=F_4/\Spin (9)$.

Here we use the standard notation from the theory of Lie groups; particularly,  
$F_4$ is one of the exceptional Lie groups in Cartan's classification. 

The indicated projective spaces $ \F P^n $ as compact Riemannian manifolds have dimensions $d$, 
\begin{equation}
d=\dim_{\bR} \F P^n=nd_0, \ d_0=\dim_{\bR}\F,
\label{eq2.1} 
\end{equation}
where $d_0=1,2,4,8$ for $\F=\bR$, $\C$, $\H$, $\bO$, correspondingly.

For spheres $S^d$ we put $d_0=d$ by definition. Projective spaces 
of dimension  $d_0$ ($n=1$) are isomorphic to the spheres $S^{d_0}$:
$\bR P^1 \approx S^1, \C P^1 \approx S^2,  \H P^1 \approx S^4, \bO P^1 \approx S^8$.
We can conveniently agree that $d>d_0$ ($n\ge 2)$ for projective spaces,
while the equality $d=d_0$ holds only for spheres. Under this convention,
the dimensions $d=nd_0$ and $d_0$ define uniquely (up to isomorphism) 
the corresponding two-point homogeneous space which we denote by $Q=Q(d,d_0)$.
We write  $\mu$ for the $G$-invariant Riemannian measure
on $Q(d,d_0)$ normalized by $\mu(Q(d,d_0)) = 1$ and $\theta$ for the metric
proportional to the corresponding $G$-invariant Riemannian distance on $Q(d,d_0)$ with
the coefficient of proportionality defined by $\diam(\theta, Q(d,d_0)) = \pi $,
see \eqref{eq1.1}.
In what follows we always assume that $n=2$ if $\F=\bO$, since  
projective spaces $\bO P^n$ do not exist for $n>2$. In more detail the geometry 
of spaces  $\F P^n$ will be outlined in section 4.

Any space $Q(d,d_0)$ is distance-invariant and the volume of balls is given by 
\begin{equation}
\left.
\begin{aligned}
&v_r=\kappa(d,d_0)\int\limits^r_0(\sin\frac{1}{2}u)^{d-1}(\cos \frac{1}{2}u)^{d_0-1}\,du, 
\quad r\in [0,\pi],
\\
&\kappa(d,d_0)=B(d/2,d_0/2)^{-1}=\frac{\Gamma(d/2+d_0/2)}{\Gamma(d/2)\Gamma(d_0/2)}.
\end{aligned}
\right\}
\label{eq2.2}
\end{equation}
Here $B(\cdot ,\cdot)$ and $\Gamma(\cdot)$ are the beta and gamma functions, and 
$v_{\pi}=\mu(Q(d,d_0))=1$. Different equivalent forms of the relation \eqref{eq2.2}
can be found in the literature, see \cite[pp.~177--178]{19}, \cite[pp.~165--168]{22},  
\cite[pp.~508--510]{24}.

From the formula \eqref{eq2.2} we obtain the following two-side bounds 
\begin{equation}
v_r\simeq r^d, \ v'_r=1-v_r\simeq (\pi-r)^{d_0}, \ r\in [0,\pi].
\label{eq2.3}
\end{equation}

To simplify notation we write in some formulas $A\lesssim B$ instead of 
$B=O(A)$, $A\gtrsim B$
instead of $B=O(A)$, and $A\simeq B$ if $A=O(B)$ and $B=O(A)$.

The {\it chordal metric} on the spaces $Q(d,d_0)$ can be defined by
\begin{equation}
\tau(x_1,x_2)=\sin \frac{1}{2}\theta(x_1,x_2), \quad x_1,x_2\in Q(d,d_0).
\label{eq2.4}
\end{equation}
Notice that the expression \eqref{eq2.4} defines a metric because the function 
$\varphi(\theta)=\sin \theta/2$, $0\le \theta\le \pi$, 
is concave, increasing and  $\varphi(0)=0$, that implies the triangle 
inequality.
For the sphere $S^d=\{x\in \bR^{d+1}:\|x\|=1\}$ we have
\begin{equation}
\left.\begin{aligned}
&\cos \theta(x_1,x_2)=(x_1,x_2), x_1,x_2\in S^d,
\\
&\tau(x_1,x_2)=\sin \frac12\theta(x_1,x_2)=\frac{1}{2}\|x_1-x_2\|,
\end{aligned}
\right\}
\label{eq2.5}
\end{equation}
where $(\cdot,\cdot)$ is the inner product and $\|\cdot \|$ is the Euclidean  
norm in  $\bR^{d+1}$.

Each projective space $\F P^n$ can be canonically imbedded into the unit sphere
\begin{equation}
\Pi:Q(d,d_0)\ni x\to \Pi(x)\in S^{m-1}\subset \bR^m, \ 
m=\frac{1}{2}(n+1)(d+2),
\label{eq2.6}
\end{equation}
such that
\begin{equation}
\tau(x_1,x_2)=\frac{1}{\sqrt{2}}\|\Pi(x_1)-\Pi(x_2)\|, \quad x_1,x_2\in \F 
P^n,
\label{eq2.7}
\end{equation}
where $\|\cdot \|$ is the Euclidean norm in $\bR^{m+1}$.
Hence, the metric $\tau(x_1,x_2)$ coincides with the Euclidean length of a segment
joining the corresponding points $\Pi(x_1)$ and $\Pi(x_2)$ on the unit sphere 
and normalized by $\diam (Q(d,d_0),\tau)=1$.
The imbedding \eqref{eq2.6} will be described explicitly in Section~\ref{sec5}.

The chordal metric $\tau$ on the complex projective space $\C P^n$ 
is known as the Fubini--Study metric. 
In connection with special point configurations in two-point homogeneous spaces
the chordal metric on projective spaces has been discussed in the papers \cite{13, 14},
see also the paper \cite{15}, where the chordal metric has been defined for  
Grassmannian manifolds.

Now we are in position to state our main results. First of all, we consider
the $L_2$-invariance principles for two-point homogeneous spaces. 
A careful analysis of the imbedding \eqref{eq2.6} leads to the following.

\begin{theorem}\label{thm2.1}
For any space $Q=Q(d,d_0)$ the chordal metric
\eqref{eq2.4} and the symmetric difference metric \eqref{eq1.13} are related by
\begin{equation}
\tau(x_1,x_2)=\gamma(Q)\,\, \theta^{\Delta}(\eta^{\natural} ,x_1,x_2), \ 
x_1,x_2\in Q,
\label{eq2.8}
\end{equation}
where $\eta^{\natural}(r)=\sin r$, $r\in [0,\pi]$, and
\begin{equation}
\gamma(Q)=\frac{\langle \tau\rangle}{\langle 
\theta^{\Delta}(\eta^{\natural})\rangle}=\frac{\diam 
(Q,\tau)}{\diam(Q,\theta^{\Delta}(\eta^{\natural}))}.
\label{eq2.9}
\end{equation}
\end{theorem} 
              
The proof of Theorem~\ref{thm2.2} is given in Section~\ref{sec6}.
It is clear that the equalities \eqref{eq2.9} follow immediately from \eqref{eq2.8}. 
It suffices to calculate the average values \eqref{eq1.12} of both metrics 
in \eqref{eq2.8} to obtain 
the first equality in \eqref{eq2.9}. Similarly, writing \eqref{eq2.8} for any pair of 
antipodal points $x_1$, $x_2$, $\theta(x_1,x_2) = \pi$, we obtain the second 
equality in \eqref{eq2.9}. Recall that points $x_1, x_2$ are antipodal for a metric
$\rho$ if $\rho(x_1,x_2) = \diam(Q,\rho)$. If points $x_1, x_2$ are antipodal for
the metric $\theta$, then in view of \eqref{eq2.4} and \eqref{eq2.8} they are also
antipodal for the metrics $\tau$ and $\theta^{\Delta}(\eta^{\natural})$.

Comparing Theorems~1.1 and 2.1, we arrive at the following.

\begin{corollary}\label{cor2.1}
(The $L_2$-invariance principle). For 
any space $Q=Q(d,d_0)$ we have the relation
\begin{equation}
\gamma(Q)\lambda [\eta^{\natural},\D_N]+\tau [\D_N]=\langle \tau\rangle 
N^2,
\label{eq2.10}
\end{equation}
where $\D_N\subset Q$ is an arbitrary $N$-point subset.

Particularly, for any $N$ we have the equality 
\begin{equation}
\gamma(Q)\lambda_N(\eta^{\natural})+\tau_N=\langle \tau\rangle N^2.
\label{eq2.11}
\end{equation}
\end{corollary}

Notice that for the sphere $S^d$ the discrepancy $\lambda [\eta^{\natural},\D_N]$ with 
the special weight function $\eta^{\natural}(r)=\sin r$ can be written in the form 
\begin{equation}
\lambda[\eta^{\natural},\D_N]=\int\limits^1_{-1}dz\int\limits_{S^d}[\#\{B(y,z)\cap 
\D_N\}-N\mu(B(y,z))]^2\,d\mu(y),
\label{eq2.12}
\end{equation}
where 
$
B(y,z)=\{x\in S^d:\cos \theta(x,y)\ge z\}, \ y\in S^d, z\in [-1,1],
$
is a spherical cap; in our notation $B(y,z)=B_r(y)$, $z=\cos r$.

For spheres the invariance principle \eqref{eq2.10} was established by 
Stolarsky~\cite{33}, see also the recent papers 
\cite{8, 11}, where the original proof of this relation 
was essentially simplified. 
Corollary~\ref{cor2.1} can be thought of as an extension of Stolarsky's invariance principle to projective spaces.
 
A metric space $\M$ with a metric $\rho$ is called isometrically
$L_q$-embeddable, if there exists a map $\varphi:\M\ni x\to \varphi(x)\in L_q$, such that
$\rho(x_1,x_2)=\|\varphi(x_1)-\varphi(x_2)\|_{L_q}$ for all $x_1$, 
$x_2\in \M$. Notice that the  $L_2$-embeddability is stronger and
implies the $L_1$-embeddability, see~\cite[Sec.~6.3]{17}. 

A compact metric space $\M$ is isometrically
$L_1$-embeddable with respect to any symmetric difference metric $\theta_r^{\Delta}$ 
and $\theta^{\Delta}(\eta)$, see \eqref{eq1.17} and \eqref{eq1.13}. 
At the same time, 
the two-point homogeneous space $Q(d,d_0)$ is isometrically $L_2$-embeddable with respect to the chordal metric $\tau$, see \eqref{eq2.5} and \eqref{eq2.7}. 
This explains our terminology of $L_1$- and $L_2$-invariance principles.

It would be interesting to find out whether there are other weight functions 
$ \eta \ne \eta^{\natural} $
for which the spaces $Q(d,d_0)$ with the metric $\theta^{\Delta}(\eta)$ are 
also $L_2$-embeddable.

Now we consider best possible bounds for the extremal quantities \eqref{eq1.9} 
and \eqref{eq1.11}. 
At first, we state in Lemma~2.1 some important auxiliary results.
Introduce the following classes of weight functions $\eta(r)$, $r\in [0,\pi]$,
\begin{equation}
\left.\begin{aligned}
W(a,b)&=\{\eta\ge 0:\|\eta\|_{a,b}<\infty\}, \ a\ge b\ge 1,
\\
\|\eta\|_{a,b}&=\int\limits^{\pi}_0(\sin\frac{1}{2}r)^{a-1}(\cos 
\frac{1}{2}r)^{b-1}\eta(r)\,dr.
\end{aligned}
\label{eq2.13}
\right\}
\end{equation}
It is worth noting that weight functions in the classes \eqref{eq2.14} admit rather 
large singularities at points  $r=0$ and $r=\pi$.

\begin{lemma}\label{lem2.1}
For any space $Q(d,d_0)$ the following hold {\rm:}

{\rm(i)} The kernel \eqref{eq1.6} and the metric \eqref{eq1.14} satisfy the bounds
\begin{equation}
\left.\begin{aligned}
|\lambda_r(y_1,y_2)|&\le C(\sin\frac{1}{2}r)^d(\cos \frac{1}{2}r)^{d_0},
\\
\theta^{\Delta}_r(y_1,y_2)&\le C(\sin \frac{1}{2}r)^d(\cos 
\frac{1}{2}r)^{d_0}.
\end{aligned}
\label{eq2.14}
\right\}
\end{equation}
If $\eta\in W(d+1,d_0+1)$, then the kernel  {\rm(1.8)} and the metric {\rm(1.13)} 
satisfy the bounds
\begin{equation}
\left.\begin{aligned}
|\lambda(\eta,y_1,y_2)|&\le C\|\eta\|_{d+1,d_0+1},
\\
\theta^{\Delta}(\eta,y_1,y_2)&\le C\|\eta\|_{d+1,d_0+1}.
\end{aligned}
\label{eq2.15}
\right\}
\end{equation}

{\rm(ii)} The metric \eqref{eq1.14} satisfies the bound
\begin{equation}
\theta^{\Delta}_r(y_1,y_2)\le C(\sin\frac{1}{2}r)^{d-1}(\cos 
\frac{1}{2}r)^{d_0-1}\theta(y_1,y_2).
\label{eq2.16}
\end{equation}
If $\eta\in W(d,d_0)$, then the metric \eqref{eq1.13} satisfies the bound
\begin{equation}
\theta^{\Delta}(\eta,y_1,y_2)\le C\|\eta\|_{d,d_0}\theta(y_1,y_2).
\label{eq2.17}
\end{equation}

Constants in the bounds \eqref{eq2.14} -- \eqref{eq2.17} depend 
only on $d$ and $d_0$.
\end{lemma}  
        
The proof of Lemma~\ref{lem2.1} is given in Section~\ref{sec7}.
It follows from Lemma~\ref{lem2.1} that the $L_1$-invariance principles 
\eqref{eq1.30} -- \eqref{eq1.32} 
hold for the spaces $Q(d,d_0)$ with weight functions $\eta\in W(d+1,d_0+1)$.

Our result on the extremal quantities\eqref{eq1.9} and \eqref{eq1.11} 
can be stated as follows.

\begin{theorem}\label{thm2.2}
For any space $Q(d,d_0)$ the following hold {\rm:}
 
If $\eta\in W(d,d_0)$, $\eta\ne 0$, then for any $N$ we have
\begin{align}
\langle \theta^{\Delta}(\eta)\rangle 
N^2-c(\eta)N^{1-\frac{1}{d}}&>\theta^{\Delta}_N(\eta)>\langle 
\theta^{\Delta}(\eta)\rangle N^2-C(\eta)N^{1-\frac{1}{d}},
\label{eq2.18}
\\
c_1(\eta)N^{1-\frac{1}{d}}&<\lambda_N(\eta)<C_1(\eta)N^{1-\frac{1}{d}}
\label{eq2.19}
\end{align}
with positive constants independent of $N$.
Particularly, for the chordal metric $\tau$ on $Q(d,d_0)$, we have
\begin{equation}
\langle \tau\rangle N^2-cN^{1-\frac{1}{d}}>\tau_N>\langle \tau\rangle 
N^2-CN^{1-\frac{1}{d}}
\label{eq2.20}
\end{equation}
with the constants $c=c(\eta^{\natural})$ and $C=C(\eta^{\natural})$.
\end{theorem}

For the chordal metric $\tau$ on the sphere $S^d$ the bounds \eqref{eq2.20} were 
known earlier. The right bound in \eqref{eq2.20} was established by Alexander \cite{1}
and the left by Beck \cite{5}, see also \cite{6, 8a}. 
Theorem~\ref{thm2.2} can be thought of as an extension of the results 
of these authors to projective spaces. However, it should be pointed out that
the bounds \eqref{eq2.18} and \eqref{eq2.19} are new even in the case of sphere $S^d$.

The proof of Theorem~\ref{thm2.2} is given in Section~\ref{sec11}.
It is clear that the right bounds in \eqref{eq2.18} and \eqref{eq2.19} follow immediately
from Theorem~\ref{thm1.2}(ii) and Lemma~\ref{lem2.1}(ii). 
In Section~\ref{sec11} we shall prove the left bound in \eqref{eq2.19}. 
By the invariance principle \eqref{eq1.32} this will immediately imply  
the left bound in \eqref{eq2.18}.
The proof of the left bound in \eqref{eq2.19} is relying on  the theory
of spherical functions on homogeneous spaces $Q(d,d_0)$. 


\section{Applications to $t$-designs}\label{sec3}

Many specific point configurations on spheres and other two-point homogeneous 
spaces are described in the literature, see, for example,  
\cite{4, 9, 10, 10a, 12, 13, 14, 15, 17a, 24, 29}.
One can ask whether the points of such specific configurations are distributed 
uniformly in the corresponding spaces, and whether  
the quadratic discrepancies \eqref{eq1.7} and the sums of pairwise 
distances \eqref{eq1.10} can be estimated precisely for such point subsets.

In the present paper we consider these questions for $t$-designs. 
Consider a smooth compact connected $d$-dimensional manifold $\M$ in $\bR^m$ 
equipped with a smooth Riemannian sructure, so that 
the geodesic distance $\theta$ and the Riemanniam measure $\mu$ normalized 
by \eqref{eq1.1} are defined on $\M$. 
An $N$-point subset $\D_N\subset \M$ is called a  
\emph {$t$-design}, if the exact quadrature formula 
\begin{equation}
\sum\limits_{x\in \D_N} F(x) = N\int\limits_{\M} F(y)\, d\mu (y)
\label{eq3.1}
\end{equation}
holds for all polynomials  $F(x), \, x\in \bR^m$ 
of total degree not exceeding $t$. 

It is known, see, for example, \cite{17a}, that any $N$-point 
$t$-design $\D_N\subset \M$ 
satisfies the bound $N\gtrsim t^d$ with a constant independent of $N$ and $t$. 
An $N$-point $t$-design $\D_N\subset \M$ is called
\emph{optimal}, if  
\begin{equation}
c_+t^d \ge N \ge c_-t^d 
\label{eq3.6}
\end{equation}
with some positive constants   $c_+$ and $c_-$ independent of $N$ and $t$. Actually,
in this definition we deal with sequences of $N$-point $t$-designs $\D_N$ as $N\to \infty$. 

As we mentioned earlier, the two-point homogeneous spaces $Q(d, d_0)$ can be 
canonically imbedded into  $\bR^m$, see the comments to \eqref{eq2.6}, 
Hence, the above definitions can be used for $Q(d, d_0)$.

Since the spaces $Q(d, d_0)$ are homogeneous, an equivalent definition of $t$-designs
can be given in the following invariant form, see \cite{4, 24}. An  $N$-point subset 
$D_N\subset Q(d,d_0)$ 
is a $t$-design, if and only if the exact quadrature formula 
\begin{equation}
\sum\limits_{x_1,x_2\in \D_N} f(\cos \theta (x_1,x_2))  = N^2
\iint\limits_{Q\times Q} f(\cos \theta (y_1,y_2))  \, d\mu (y_1)\, d\mu
(y_2)
\label{eq3.2}
\end{equation}
holds for all polynomials $f(z), \, z\in \C$, of degree not exceeding $t$.
The formula \eqref{eq3.2} is equivalent to the following quadrature formulas 
\begin{equation}
\sum\limits_{x\in \D_N} f(\cos \theta (x,y))  = N
\int\limits_{Q} f(\cos \theta (x,y))  \, d\mu (x), 
\label{eq3.3}
\end{equation}
which holds identically for all $y\in Q$. Another equivalent definition of 
$t$-designs can be given in terms of spherical functions on 
the spaces $Q(d,d_0)$, see \cite{4, 24}. 
We shall return to these questions in Section~\ref{sec8}, see \eqref{eq8.33}.

For any $N$-point subset $\D_N\subset \M$ we put 
\begin{equation}
\nu [\D_N,r] =\max\limits_{y\in Q}\# 
\{ B_r(y)\cap \D_N\}, \quad r\in [0,\pi],
\label{eq3.7}
\end{equation}
and $\nu [\D_N,r]=N$ if $r>\pi$. 

Our result on $t$-designs can be stated as follows. 

\begin{theorem}\label{thm3.1}
Let the weight function  
$\eta \in W(d,d_0)$, then the following hold:  

{\rm (i)} There exists a constant $L\ge 1$ depending only on $d$ and $d_0$,
such that for any $N$-point  $t$-design $\D_N\subset Q(d,d_0)$ with
$t\ge 2L/\pi$ we have 
\begin{equation}
\la [\eta, \D_N]  < C t^{d-1}  (\nu [\D_N, Lt^{-1}])^2.
\label{eq3.8}
\end{equation}

{\rm (ii)} For optimal $N$-point $t$-designs $\D_N\subset Q (d,d_0)$ the bound 
\eqref{eq3.8} takes the form 
\begin{equation}
\la [\eta,\D_N] < C N^{1-\frac 1d} 
(\nu [\D_N, c^{-1/d}_{+} LN^{-1/d}])^2,
\label{eq3.9}
\end{equation}
where $c_+$ is the constant in the definition {\rm (3.6)}. 

The constants $C$ in the bounds \eqref{eq3.8} and \eqref{eq3.9} depend only on $d$,
$d_0$ and $\eta$. 
\end{theorem}

The inequality \eqref{eq3.9} follows immediately 
from \eqref{eq3.8} and the definition \eqref{eq3.6}.  
The proof of the bound  \eqref{eq3.8} is given in Section~\ref{sec11}.
The proof is relying
on the theory of spherical functions on homogeneous spaces $Q(d,d_0 )$. 

We are interested whether the factor in \eqref{eq3.9} with the function $\nu$
can be bounded from above
by a constant independent of $N$. In this case, the order of bound \eqref{eq3.9}
would be the best possible. Two simple sufficient conditions for this
are given below in Lemma 3.1.

Introduce some definitions.
For an arbitrary $N$-point subset $\D_N\subset \M$, we put 
\begin{equation}
\delta [\D_N]= \frac 12 \min \{ \theta (x_1,x_2): x_1,x_2 \in \D_N,
x_1\ne x_2\}  
\label{eq3.10}
\end{equation}
The balls $B_{\delta}(x)$, $\delta=\delta[\D_N]$, $x\in \D_N$,
do not overlap. Therefore, $\delta \lesssim N^{-1/d}$, since
the volume of balls  $v_r(x)\simeq r^d$ uniformly for $r\in [0,\pi]$
and $x \in \M$. 
An $N$-point subset $\D_N\subset \M$ is called 
\emph{well-separated}, if $\delta [\D_N]\ge cN^{-1/d}$ with a constant
$c>0$ independent of $N$. 

Consider an equal-measure partition $\mathcal{P}_N = \{P_i\}_{1}^{N}$ of 
the manifold $\M$,
$$
\M\ =\bigcup^{N}_{i=1} P_i,
\quad
\mu(P_i\cap P_j)=0,
\quad
i\ne j,
\quad
\mu(P_i) = 1/N,
$$
and put 
$$
\Diam(\mathcal{P}_N, \theta)=\max_{1\le i\le N}\diam(P_i,\theta).
$$
We say that an equal-measure partition $\mathcal{P}_N$ is of \emph{small diameter},
if 
\begin{equation}
\Diam(\mathcal{P}_N, \theta) \le c_0 N^{-1/d} 
\label{eq3.10a}
\end{equation}
with a constant $c_0 > 0$ independent of $N$.  
Constructions of equal-measure partitions of small diameter are known for
a large class of smooth compact manifolds in $\bR^m$, see  \cite{19a} and references
therein.

We also say that an $N$-point subset $\D_N = \{x_i\}_1^N \in \M$ is 
\emph{subordinated}
to a partition $\mathcal{P}_N = \{P_i\}_1^N$ of $\M$, if $x_i \in P_i, i = 1\dots N$. 

We conveniently agree that for $r>\pi$ the ball $B_r(x)=\M$ and $v_r(x)=1$.  
With these convention and definitions the following result is true.
\begin{lemma}\label{lem3.1}
Suppose that an $N$-point subset $\D_N\subset \M$ satisfies one of 
the following conditions{\rm:}

{\rm(i)} $\D_N$ is well-separated,

{\rm(ii)} $\D_N$ is subordinated to an equal-measure partition of small diameter.

Then, for any constant $c>0$ there exists a constant $C=C(c)$ independent 
of $N$ such that 
\begin{equation}
\nu [\D_N,  cN^{-1/d}] \le C.
\label{eq3.11}
\end{equation}
\end{lemma}  

\begin{proof}[Proof] For brevity, we write $a=cN^{-1/d}$. Consider 
the ball $B_a(y)$ centered at an arbitrary point $y\in Q$ and put $\E=B_a(y)\cap \D_N$,
$K=\# \{ \E\}$. Assume also that points of $\D_N=\{x_i\}_1^N$
are enumerated such that $\E=\{x_i\}_1^K$.

(i) By the definition of a well-separated subset $\D_N$ , the balls $B_{\delta}(x)$, 
$\delta=\delta [\D_N]$, $x\in \E$, do not overlap and all these balls are contained in the ball 
$B_{a+\delta}(y)$. Therefore, 
$
\sum\nolimits_{1\le i \le K}v_{\delta}(x_i)\le v_{a+\delta}$. 
Since  $v_r(x)\simeq r^d,
$ 
we have
$
K\lesssim v_{a+\delta}/v_{\delta} \simeq  
( 1 + C / c )^d, 
$
and \eqref{eq3.11} follows. 

(ii) By the definition of a subset $\D_N$ subordinated to an equal-measure
partition $\mathcal{P}_N = \{P_i\}_1^N$ of small diameter $b=c_1N^{-1/d}$,
each part $P_i, 1\le i \le K$ is contained in the ball $B_{a+b}(y)$. Therefore,
$N^{-1}K\le v_{a+b}(y)$, and $K \le Nv_{a+b}(y) \simeq (c+c_1)^d$, that proves
\eqref{eq3.11}.
\end{proof}

Comparing Theorem 3.1 with Lemma 3.1, and taking into account the left
bounds of Theorem 2.2, we arrive at the following.

\begin{corollary}\label{cor3.1}  Let the weight function
$\eta \in W(d,d_0)$, $\eta \ne 0$. Suppose that an $N$-point subset
$\D_N\subset Q (d,d_0)$ is an optimal $t$-design and satisfies 
one of the conditions (i) or (ii) of Lemma 3.1.
Then, for all sufficiently large $N$ we have 
\begin{align}
\langle \theta^{\Delta}(\eta)\rangle N^2 -cN^{1-\frac 1d} & >
\theta^{\Delta}[\eta, \D_N] > \langle \theta^{\Delta}(\eta)\rangle N^2
-CN^{1-\frac 1d},
\label{eq3.12}
\\
cN^{1-\frac 1d} & < \la [\eta, \D_N]  <CN^{1-\frac 1d}
\label{eq3.13}
\end{align}

Particularly, for the chordal metric $\tau$ on $Q(d,d_0)$ we have 
\begin{equation}
\langle r\rangle N^2 -cN^{1-\frac 1d} >\tau [\D_N] > \langle r\rangle N^2 
-C N^{1-\frac 1d}
\label{eq3.14}
\end{equation}

The positive constants in \eqref{eq3.12} -- \eqref{eq3.14} are independent of $N$. 
\end{corollary}

The existence of optimal $t$-designs was a long standing open problem 
known as the Korevaar--Meyers conjecture. In the papers \cite{9, 10} by Bondarenko,
Radchenko and Viazovska  a breakthrough on the problem was obtained for 
spherical $t$-designs. In \cite{9} the existence of optimal $t$-designs
$\D_N \in S^d$ 
was proved for all sufficiently large $N$, and it was proved
in \cite{10} that such optimal $t$-designs can be chosen as well-separated
subsets on the spheres $S^d$. Hence, Corollary 3.1 is applicable for 
the spheres $S^d$. 

Using optimal spherical $t$-designs $\D_N$ one can easily construct 
optimal $[t/2]$-designs $\D^\circ_N$ on 
the real projective space $\bR P^d=Q(d,1)$. Furthermore, if $\D_N$ is well-separated
on $S^d$, then $\D^\circ_N$ satisfies the relation \eqref{eq3.11} on $\bR P^d$.
Hence, Corollary 3.1 is also applicable for the real
projective spaces $\bR P^d$. 

The corresponding generalizations to the projective spaces
 $\C P^n$, $\H P^n$  and  $\Q P^2$
are not so straightforward. In the recent paper \cite{17a} by Etayo, Marzo
and Ortega--Cerd\`{a} the results of the paper \cite{9} were
extended to smooth compact connected algebraic manifolds
$
\M=\{x \in \bR^m: f_1(x)=\dots =f_r(x)=0\}
$, 
where $f_1, \dots , f_r$ are
polynomials with real coefficients.  
We state results from \cite{17a} in the following form. 

\begin{theorem}\label{thm3.2} 
Let $\M$ be a smooth compact connected $d$-dimensional affine algebraic manifold  
in $\bR^m$ equipped with a smooth Riemannian structure. Then 
there exist the positive constants $c_-,\, c_+$ and $c_0$ 
depending only on $\M$, such that the following is true.

{\rm(i)} For all sufficiently large $N$ there exist $N$-point optimal $t$-designs 
$\D_N \subset \M$ satisfying \eqref{eq3.6}.

{\rm(ii)} Each of these optimal $t$-designs $\D_N$ is subordinated to 
an equal-measure partition $\mathcal{P}_N $ of small diameter on $\M$
satisfying \eqref{eq3.10a}.
 
\end{theorem}

In fact, the statement (i) of Theorem 3.2 is contained in Theorem 2.2 in \cite{17a},
while the statement (ii) follows immediately from 
the proof of Theorem 2.2 in \cite{17a}.

The two-point homogeneous spaces $Q(d,d_0)$ can be realized as
smooth compact connected affine algebraic manifolds. For the spheres $S^d$ 
this is obvious, while for the projective spaces 
$\bR P^n, \C P^n,  \H P^n$ and the projective octonionic plane $\bO P^2$ this 
is also well known and follows immediately from the explicit formulas \eqref{eq5.13} and \eqref{eq5.14}.

Comparing Theorem 3.2 and Corollary 3.1, we arrive at the following.

\begin{corollary}\label{cor3.2}
On each space $Q(d,d_0)$ for all sufficiently large $N$ there exist $N$-point 
optimal $t$-designs $\D_N \subset Q(d,d_0)$, which satisfy the relations
\eqref{eq3.12}, \eqref{eq3.13}, \eqref{eq3.14} of Corollary 3.1.

\end{corollary}


\section*{II. Geometry of two-point homogeneous spaces and \\ the 
$L_2$-invariance principles}

\section{Preliminaries: Models of projective spaces and chordal metrics }\label{sec5}

In this section we define the chordal metrics on the projective spaces
$\F P^n$, $\F=\bR$, $\C,$ $\H$, $n\ge 2$, and the octonionic projective 
plane  $\bO P^2$ in terms of special models for these spaces. 
For the sake of convenience, we describe such
models in sufficient detail and give the necessary references.

Recall the general facts on the algebras $\bR, \C, \H, \bO$ over the field of real numbers. 
We have the natural inclusions $\bR\subset \C\subset \H\subset \bO.$
where the octonions $\bO$ are a nonassociative and noncommutative algebra
of  dimension 8  with a basis $1, e_1, e_2, e_3, e_4, e_5, e_6, e_7$ 
(their multiplication table can be found in \cite[p.~150]{3} and \cite[p.~90]{7}),
the quaternions $\H$ are an associative but noncommutative subalgebra of
dimension 4 spanned by ${1,e_1, e_2, e_3}$, finally, $\C$ and $\bR$ are
associative and commutative subalgebras of dimensions 2 and 1 spanned by 
${1, e_1}$ and ${1}$, correspondingly. 
From the multiplication table one can easily see that for any two indexes
$1 \le i,j \le 7,i\ne j, $ there exists an index $1 \le k \le 7$, such that
\begin{equation}
e_ie_j=-e_je_i=e_k, \quad i\ne j, \quad e^2_i=-1.
\label{eq5.1}
\end{equation}
Let $a=\alpha_0+\sum\nolimits^{7}_{i=1}\alpha_i e_i\in \bO$, 
$\alpha_i\in \bR$, $0\le i\le 7$, be a typical octonion. We write  $\RE a=\alpha_0$  
for the real part,  
$\bar a=\alpha_0-\sum\nolimits^{7}_{i=1} \alpha_ie_i$ for the conjugation, 
$|a|=\big (\alpha^2_0+\sum\nolimits^{7}_{i-1}\alpha^2_i\big)^{1/2}$ fot the norm. 
Using \eqref{eq5.1}, one can easily check that 
\begin{equation}
\RE ab=\RE ba,\quad \overline{ab}=\overline{ba}, \quad 
|a|^2=a\bar a=\bar aa, \quad |ab|=|a|\,|b|. 
\label{eq5.2}
\end{equation}
The last equality in \eqref{eq5.2} implies that all algebras $\bR, \C, \H, \bO$  
are division algebras.
Notice also that by a theorem of Artin a subalgebra in $\bO$ generated 
by any two octonions is associative and isomorphic to one of the algebras 
$\H$, $\C$, or $\bR$, see \cite{3}.
 
The standard model of projective spaces over
the associative algebras $\F=\bR$, $\C$, $\H$ is well known, see, for example,  \cite{3, 7, 20, 36}. 
Let $\F^{n+1}$ be a linear
space of vectors $\mathbf a=(a_0,\dots,a_n)$, $a_i\in \F$, $1\le i\le n$
with the right multiplication by scalars $a\in\F$, the Hermitian
inner product 
\begin{equation}
(\mathbf a,\mathbf b) =\sum\limits^{n}_{i=0} \bar a_i b_i ,\quad
\mathbf a,\mathbf b \in \F^{n+1},
\label{eq5.3}
\end{equation}
and the norm $|\mathbf a|$,
\begin{equation}
|\mathbf a|^2=(\mathbf a, \mathbf a) =\sum\limits^{n}_{i=0} |a_i|^2. 
\label{eq5.4}
\end{equation}

In this case, in view of associativity of the algebras $\F=\bR$, $\C,\H$,  
a projective space $\F P^n$ can be defined as a set of 
one-dimensional (over  $\F$) subspaces in $\F^{n+1}$: 
\begin{equation}
\F P^n=\{ p(\mathbf a)=\mathbf a \F : \mathbf a \in \F^{n+1}, \, 
|\mathbf a|=1\} .
\label{eq5.5}
\end{equation}
The metric $\theta$ on $\F P^n$ is defined by 
\begin{equation}
\cos \frac 12 \theta (\mathbf a,\mathbf b)\!=\!|(\mathbf a,\mathbf b)|, \quad
\mathbf a,\mathbf b\in \F^{n+1}, \quad  |\mathbf a|\!=\!|\mathbf b|\!=\!1, \,\,  
0\le \theta (\mathbf a,\mathbf b)\le \pi,
\label{eq5.6}
\end{equation}
i.e.  $\frac 12 \theta (\mathbf a,\mathbf b)$ is the angle between  
the subspaces $p(\mathbf a)$ and $p(\mathbf b)$.
The transitive group of isometries   $U(n+1, \F)$ for the metric $\theta$ consists of 
nondegenerate linear transformations of the space  $\F^{n+1}$, preserving  
the inner product  \eqref{eq5.3}, and the stabilizer of a point is isomorphic to
the subgroup $U(n,\F)\times U(1,\F)$. Hence, 
\begin{equation}
\F P^n=U(n+1, \F)/ U(n,\F) \times U(1,\F).
\label{eq5.7}
\end{equation}
The groups $U(n+1,\F)$ can be easily determined (they have been indicated in section 2 
in the list of compact connected two-point homogeneous spaces). 
A Riemannian $U(n+1,\F)$-invariant structure corresponding to the metric $\theta$
can be also defined on the projective space \eqref{eq5.5}, and one can easily check that
these spaces are two-point homogeneous spaces.

There is another model where a projective space $\F P^n$, $\F=\bR, \C, \H$, is identified
with the set of orthogonal projectors onto the one-dimensional subspaces in  $\F^{n+1}$. 
This model admits a generalization to the octonionic projective plane $\bO P^2$ 
and in its terms the chordal metric can be naturally defined for all projective spaces.

Let $\HH (\F^{n+1})$ denote the set of all Hermitian ${(n+1)}\times (n+1)$ matrices 
with the entries in $\F$, $\F=\bR$, $\C,\H,\bO$, 
\begin{equation}
\HH (\F^{n+1})= \{ A=((a_{ij})) : a_{ij}=\overline{a}_{ji}, \;
a_{ij}\in \F, \, 0\le i,j\le n \}  
\label{eq5.8}
\end{equation}
with $n=2$ if $\F=\bO $.
It is clear that $\HH(\F^{n+1})$ is a linear space over $\bR$ of dimension 
\begin{equation}
m=\dim_{\bR}\HH (\F^{n+1})=\frac 12 (n+1) (d+2), \quad d = n d_0.
\label{eq5.9}
\end{equation}

The linear space $\HH (\F^{n+1})$ is equipped  with the symmetric real-valued inner product 
\begin{equation}
\langle A,B\rangle =\frac 12 \Tr (AB+BA)= \RE \Tr AB=
\RE \sum\limits^{n}_{i,j=0} a_{ij}\overline{b_{ij}}
\label{eq5.10}
\end{equation}
and the norm 
\begin{equation}
\| A\| =(\Tr A^2)^{1/2} =  \left(  \sum\limits^{n}_{i,j=0}
|a_{ij}|^2\right)^{1/2},
\label{eq5.11}
\end{equation}
here $\Tr A=\sum\nolimits^{n}_{i=0}a_{ii}$  denotes the trace of a matrix $A$. 
For the distance $\| A-B\|$ between matrices $A,B\in \HH (\F^{n+1})$, we have 
\begin{equation}
\| A-B\|^2 =\| A\|^2 +\| B\|^2 -2 \langle A,B\rangle.
\label{eq5.12}
\end{equation}
Thus, $\HH (\F^{n+1})$ can be thought of as the $m$-dimensional Euclidean space. 

If $\F\ne \bO$, the orthogonal projector  $\Pi_{\mathbf a}\in \HH (\F^{n+1})$ onto
a one-dimensional subspace  
$p(\mathbf a)=\mathbf a \F$, $\mathbf a=(a_0,\dots,a_n)\in \F^{n+1}$,
$|\mathbf a|=1$, can be given by $\Pi_{\mathbf a}=\mathbf a (\mathbf a,\cdot)$
or in the matrix form 
$\Pi_{\mathbf a}=[a_i\bar a_j]$, $0\le i,j\le n$. Therefore, the projective space 
\eqref{eq5.5} can be written as follows
\begin{equation}
\F P^n=\{ \Pi\in \HH (\F^{n+1}): \Pi^2=\Pi, \, \,  \Tr \Pi =1\} .
\label{eq5.13}
\end{equation}
The group of isometries $U(n+1,\F)$ acts on such projectors by
the formula $g(\Pi)=g\Pi g^{-1}$, $g\in U(n+1,\F)$. 

For the octonionic projective plane $\bO P^2$ the similar model is also known.
A detailed discussion of this model can be found
in \cite{3, 7, 20} including an explanation why octonionic projective
spaces $\bO P^n$ do not exist if $n>2$.
In this model one puts by definition 
\begin{equation}
\bO P^2=\{ \Pi \in \HH (\bO^3): \Pi^2=\Pi, \, \,  \Tr \Pi=1\}.
\label{eq5.14}
\end{equation}
Thus, the formulas \eqref{eq5.13} and \eqref{eq5.14} are quite similar. One can check that each 
matrix in \eqref{eq5.14} can be written as 
$\Pi_{\mathbf a}\in \bO P^2$ for a vector $\mathbf a=(a_0,a_1,a_2)\in \bO^3$, where 
$\Pi_{\mathbf a}=[a_i\bar a_j]$, 
$0\le i,j\le 2$, $|\mathbf a|^2=|a_0|^2+|a_1|^2+|a_2|^2=1$, and additionally  
$(a_0a_1)a_2=a_0(a_1a_2)$, see  \cite[Lemma 14.90]{20}. The additional condition 
means that the subalgebra in $\bO$ generated by the coordinates $a_0,a_1,a_2$
is associative. Using this fact, one can easily show that $\bO P^2$  is
a 16-dimensional compact connected Riemannian manifold, see \cite{3, 7, 20}. 

The group of nondegenerate linear transformations $g$ of the space $\HH(\bO^3)$ 
preserving the squares  $g(A^2)=g(A)^2$, $A\in \HH(\bO^3)$, is isomorphic to
the 52-dimensional exceptional Lie group  $F_4$. This group also preserves 
the trace, inner product \eqref{eq5.10} and norm \eqref{eq5.11} of matrices 
$A\in \HH (\bO^3)$. 
The group $F_4$ is transitive on  $\bO P^2$, and the stabilizer of a point is 
isomorphic to the spinor group $\Spin (9)$, see  \cite[Lemma 14.96 and Theorem 14.99]{20}.  
Hence, $\bO P^2=F_4/\Spin (9)$ is a homogeneous space, and one can prove that
$\bO P^2$ is a two-point homogeneous space. 

Notice that the relations $\Pi^2=\Pi, \, \,  \Tr \Pi=1 $ in \eqref{eq5.13}
and \eqref{eq5.14} are polynomial equations in the corresponding 
$m$-dimensional Euclidean space $\HH (\F^{n+1})$. Hence, 
the projective spaces $\bR P^n, \C P^n,  \H P^n$ and the octonionic 
projective plane $\bO P^2$ can be thought of as affine algebraic manifolds in $\bR^m$.

For our discussion we need to describe the structure of geodesics
in projective spaces. Such a description can be easily done in terms
of models \eqref{eq5.13} and \eqref{eq5.14}. It is known, see \cite{7, 21, 36}, that 
all geodesics on a two-point homogeneous space $Q(d,d_0)$ are closed
and homeomorphic to the unit circle.  The group of isometries is transitive
on the set of geodesics and the the stabilizer of a point is transitive
on the set of geodesics passing through this point. Therefore,
all geodesics have the same length $2\pi$ (under the normalization \eqref{eq1.1} 
for the invariant Riemannian distance). 

The inclusions $\bR\subset \C\subset \H\subset \bO$ induce the following inclusions 
of the corresponding projective spaces 
\begin{equation}
\F_1P^{n_1} \subseteq \F P^n, \quad \F_1\subseteq \F, \quad n_1\le n,
\label{eq5.15}  
\end{equation}
moreover, the subspace $\F_1P^{n_1}$ is a geodesic submanifold in $\F P^n$,
see \cite[Sec.~3.24]{7}. Particularly, the real projective line $\bR P^1$, 
homeomorphic to the unit circle $S^1$, is embedded as a geodesic into
all projective spaces $\F P^n$, 
\begin{equation}
S^1 \approx \bR P^1 \subset \F P^n,
\label{eq5.16}
\end{equation}
see \cite[Proposition 3.32]{7}. In \eqref{eq5.16} $n=2$ if $\F=\bO$. These facts
can also be immediately derived from a general description of geodesic submanifolds 
in Riemannian symmetric spaces, see \cite[Chap. VII, Corollary 10.5]{21}.

Using the models \eqref{eq5.13} and \eqref{eq5.14}, we can write the real 
projective line $\bR P^1$ as the following set of $2\times 2$ matrices: 
\begin{equation}
\bR P^1=\{\zeta (u), u\in \bR/\pi \bZ\} ,
\label{eq5.17}
\end{equation}
where
\begin{equation*}
\zeta (u)  \!=\! 
\begin{pmatrix}
\cos^2 u & \sin u \cos u  \\  \sin u \cos u & \sin^2 u
\end{pmatrix} \! 
\\
=\!
\begin{pmatrix}
\cos u & -\sin u \\ \sin u  & \cos u 
\end{pmatrix}    
   \begin{pmatrix}
1 & 0 \\ 0 & 0 
\end{pmatrix}
\begin{pmatrix}
\cos u & \sin u \\ \sin u & \cos u
\end{pmatrix}.
\end{equation*}   
For each $u\in \bR$ the matrix $\zeta (u)$ is an orthogonal projector onto   
the one-dimensional subspace $x\bR$, $x=(\cos u, \sin u)\in S^1$. 
The embedding $\bR P^1$ into $\F P^n$ can be written as the following set of 
$(n+1)\times (n+1)$ matrices: 
\begin{equation}
Z= \{ Z(u), u\in \bR /\pi \bZ\} \subset \F P^n,
\label{eq5.18}
\end{equation}
where
\begin{equation*}
Z(u)= 
\begin{pmatrix}
\zeta (u) & 0_{n-1,2}  \\  0_{2, n-1}  & 0_{n-1,n-1}
\end{pmatrix},
\end{equation*}
where $0_{k,l}$ denotes the zero matrix of size $k\times l$. 
The set of matrices \eqref{eq5.18}  is a geodesic in $\F P^n$. 
All other geodesics are of the form $g(Z)$, where $g\in G$ is an isometry
of the space $\F P^n$. 
The parameter $u$ in \eqref{eq5.18} and the geodesic distance $\theta$ on the space 
$\F P^n$ are related by 
\begin{equation}
\theta (Z(u), Z(0)) =2 |u|, \quad - \pi/2 <u\le  \pi/2,        
\label{eq5.19}
\end{equation}
and for all $u\in \bR$ this formula can be extended by periodicity.
Particularly, we have
\begin{equation*}
\theta (Z(u/2), Z(-u/2)) =
\begin{cases}
2\min \{ u, \pi -u\} \, \, &  \text{if} \quad 0\le u\le \pi , \\
2u \, \,  & \text{if} \quad 0\le u \le \pi/2 .
\end{cases}
\end{equation*}
Therefore, 
\begin{equation}
\theta (Z(v), Z(-v)) =
4v,  \quad 0\le v \le \pi/4. 
\label{eq5.20}
\end{equation}
The relation \eqref{eq5.20} will be needed in the next section.

Now, we define the chordal distance on projective spaces.
The formulas \eqref{eq5.13}, \eqref{eq5.14} and \eqref{eq5.11} imply 
\begin{equation}
\| \Pi\|^2 = \Tr \Pi^2 = \Tr \Pi =1.
\label{eq5.21}
\end{equation}
for any $\Pi\in \F P^n$. Therefore, the projective spaces $\F P^n$,
defined by \eqref{eq5.13} and \eqref{eq5.14}, are submanifolds in the unit sphere 
\begin{equation}
\F P^n\subset S^{m-1} =\{ A\in \HH (\F^{n+1}) :\| A\| =1\} \subset
\HH (\F^{n+1})\approx \bR^m.
\label{eq5.22}
\end{equation}
It fact, this is an embedding of $\F P^n$ into the $(m-2)$-dimensional sphere,
the intersection of the sphere $S^{m-1}$ with the hyperplane in $\HH(\F^{n+1})$
defined by $\Tr A=1$, see \eqref{eq5.21}.

The \emph{chordal distance } $\tau (\Pi_1,\Pi_2)$  between 
$\Pi_1,\Pi_2\in \F P^n$ is defined as the Euclidean distance \eqref{eq5.12}:
\begin{equation}
\tau (\Pi_1,\Pi_2)= \frac{1}{\sqrt{2}} \| \Pi_1-\Pi_2\| =
(1-\langle \Pi_1,\Pi_2\rangle )^{1/2}.
\label{eq5.23}
\end{equation}
The coefficient $1/\sqrt{2}$ is chosen to satisfy $\diam (\F P^n, \tau)=1$. 

It is clear from \eqref{eq5.23} that  $\tau (g(\Pi_1)$, $g(\Pi_2))=\tau (\Pi_1,\Pi_2)$  
for all isometries $g\in G$ of the space $\F P^n$.
Since $\F P^n$ is a two-point homogeneous space, for any 
$\Pi_1,\Pi_2\in \F P^n$ with $\theta (\Pi_1,\Pi_2)=2u$, 
$0\le u\le \frac12 \pi$, there exists $g\in G$, such that 
$g(\Pi_1)=Z(u)$, $g(\Pi_2)=Z(0)$.  From   \eqref{eq5.23}, \eqref{eq5.18} and \eqref{eq5.17}, 
we obtain 
$
\tau (Z(u),Z(0)) =\sin u =\sin \frac12\theta (\Pi (u),\Pi(0)).
$
Therefore, 
\begin{equation}
\tau (\Pi_1,\Pi_2)= \sin \frac12 \theta (\Pi_1,\Pi_2),
\label{eq5.24}
\end{equation}
as it was defined before in \eqref{eq2.4}.

Notice also that antipodal points $\Pi_+,\Pi_-\in \F P^n$, 
i.e.  $\theta (\Pi_+,\Pi_-)=\pi$ and $\tau (\Pi_+,\Pi_-)=1$, 
can be characterized by the orthogonality condition 
$\langle \Pi_+,\Pi_-\rangle =0$, see \eqref{eq5.23}, \eqref{eq5.24}.

\section{Proof of Theorem 2.1}\label{sec6}

The proof of Theorem~\ref{thm2.1} is relying on the following special representation
of the symmetric difference metric \eqref{eq1.13}, 
given earlier in see~\cite[Lemma 2.1]{30}. Here this representation is given 
in a form adapted to the chordal metric \eqref{eq5.23}. 

\begin{lemma}\label{lem6.1} Let the weight function $\eta$ be summarized  
on the interval $[0,\pi]$, then  
\begin{equation}
\theta^{\Delta} (\eta,y_1,y_2)=\frac12 \int\limits_{\M} | \sigma
(\theta(y_1,y)) - \sigma (\theta (y_2,y)) | \, d\mu (y)
\label{eq6.1}
\end{equation}
with the nonincreasing function 
\begin{equation}
\sigma(r)= \int\limits^{\pi}_{r} \eta (u) \, du.
\label{eq6.2}
\end{equation}

Particularly, if  $\M$  is a two-point homogeneous space $Q=Q(d,d_0)$
and the weight function $\eta^{\natural}(r)=\sin r$, then   
\begin{equation}
\theta^{\Delta} (\eta^{\natural}, y_1,y_2)= \int\limits_Q 
|\tau (y_1,y)^2-\tau (y_2,y)^2| \, d\mu (y),
\label{eq6.3}
\end{equation}
where $\tau (\cdot,\cdot)$ is the chordal metric {\rm (5.23)} on $Q(d,d_0)$. 
\end{lemma}

\begin{proof}[Proof] For brevity, we write $\theta (y_1,y)=\theta_1$ and 
$\theta(y_2,y)=\theta_2$. Using \eqref{eq1.13}, \eqref{eq1.16} and \eqref{eq1.17}, we obtain 
\begin{align}
&\theta^{\Delta} (\eta ,y_1,y_2)
\notag
\\
& = \frac12 \int\limits_{\M} \left( \int\limits^{\pi}_{0} (\chi  
(r-\theta_1) +\chi (r-\theta_2) -2\chi (r-\theta_1)\chi
(r-\theta_2))  \eta(r)\, dr \right) \, d\mu (y)
\notag
\\
& = \frac 12 \int\limits_{\M} (\sigma (\theta_1) +\sigma (\theta_2) 
-2\sigma (\max \{ \theta_1,\theta_2\} )) \, d\mu (y).
\label{eq6.4}
\end{align}
Since $\sigma$ is a nonincreasing function, we have 
\begin{equation}
2\sigma (\max\{ \theta_1,\theta_2\})\! =\! 2\min \{ \sigma (\theta_1),\sigma
(\theta_2)\} 
\!= \!\sigma (\theta_1)\!+\!\sigma (\theta_2)\! -\!|\sigma (\theta_1)\!-\!\sigma 
(\theta_2)| .
\label{eq6.5}
\end{equation}
Substituting \eqref{eq6.5} into \eqref{eq6.4}, we obtain \eqref{eq6.1}. 

If  $\eta^{\natural}(r)=\sin r$, then $\sigma^{\natural}(r)= 
2-2\sin^2 r/2$. Substituting this expression into \eqref{eq6.1} and using the definition
 \eqref{eq5.24}, we obtain \eqref{eq6.3}. 
\end{proof}

For completeness, we give in the beginning a very short proof of Theorem~\ref{thm2.1} 
in the case of spheres. 
\begin{proof}[Proof of Theorem~\ref{thm2.1} for spheres]  For the sphere
$S^d$ the chordal metric $\tau$ is defined \eqref{eq2.5}. We have
\begin{align}
& \tau (y_1,y)^2-\tau (y_2,y)^2 =\frac14 (\|y_1-y\|^2-\|y_2-y\|^2) 
\notag
\\
& = -\frac12 (y_1-y_2,y) =- \tau (y_1,y_2) (x,y), \, y_1,y_2\in S^d,
\label{eq6.6}
\end{align}
where $x=\|y_1-y_2\|^{-1}(y_1-y_2)\in S^d$. Substituting 
\eqref{eq6.6} into \eqref{eq6.3}, we obtain 
\begin{equation}
\theta^{\Delta}(\eta^{\natural},y_1,y_2)= \tau (y_1,y_2)
\int\limits_{S^d} |(x,y)| \, d\mu (y).
\label{eq6.7}
\end{equation}
It is clear that the integral in \eqref{eq6.7} is independent of $x\in S^d$.
This proves the equality \eqref{eq2.8} for $S^d$ with the constant 
$
\gamma (S^d) =\left( \int\nolimits_{S^d}  |(x,y)| \,d\mu (y)\right)^{-1}. 
$
\end{proof}

\begin{proof}[Proof of Theorem 2.1 for projective spaces]
We write $\Pi_1,\Pi_2,\Pi$ for points in the models of projective spaces \eqref{eq5.13}
and \eqref{eq5.14}.
With this notation, the relation \eqref{eq6.3} takes the form  
\begin{equation}
\theta^{\Delta}(\eta^{\natural},\Pi_1,\Pi_2)= \int\limits_{\F P^n}
|\tau (\Pi_1,\Pi)^2 -\tau (\Pi_2,\Pi)^2 | \, d\mu (\Pi).
\label{eq6.8}
\end{equation} 

Since $\F P^n$ is a two-point homogeneous space, for $\Pi_1,\Pi_2 \in \F P^n$ 
with $\theta (\Pi_1,\Pi_2)=4v$, $0\le v\le \pi/4$, there exists an isometry $g\in G$, 
such that $g(\Pi_1)=Z(v)$, $g(\Pi_2)=Z(-v)$, see \eqref{eq5.20}. Therefore,
\begin{equation}
\int\limits_{\F P^n} |\tau (\Pi_1,\Pi)^2-\tau (\Pi_2,\Pi)^2 | \, d\mu 
(\Pi)   
\\
=\int\limits_{\F P^n}|\tau (Z(v),\Pi)^2-\tau (Z(-v),\Pi)^2|\,d\mu (\Pi).
\label{eq6.9}
\end{equation}
From the definition \eqref{eq5.23}, we obtain 
\begin{align}
 \tau (Z(v), \Pi)^2\!-\!\tau (Z(-v), \Pi)^2
&\! =\!  \frac12 (\| Z(v)-\Pi\|^2\! -\! \| Z(-v)\!-\!\Pi \|^2) \notag
\\
& = \langle Z(v)-Z(-v), \Pi  \rangle .
\label{eq6.10}
\end{align}
The formulas \eqref{eq5.17} and \eqref{eq5.18} imply 
$$
Z(v)-Z(-v)=
\begin{pmatrix}
\zeta(v) -\zeta(-v) & 0_{n-1,2} \\
0_{2,n-1}  &  0_{n-1,n-1}
\end{pmatrix}
$$
and 
$$
\zeta(v)-\zeta(-v)=
\begin{pmatrix}
0  &  \sin 2u \\ \sin 2u & 0
\end{pmatrix}
= \sin 2u (\zeta_+ -\zeta_-), 
$$
where 
$$
\zeta_+= \frac12 
\begin{pmatrix}
1 & 1 \\ 1 & 1
\end{pmatrix} , \quad \zeta_-=\frac12 
\begin{pmatrix}
1  &-1 \\ -1 & 1
\end{pmatrix} .
$$
Therefore, 
\begin{equation}
Z(v)-Z(-v) =\sin 2v (Z_+-Z_-),
\label{eq6.11}
\end{equation}
where
$$
Z_{\pm} = 
\begin{pmatrix}
\zeta_{\pm} & 0_{n-1,2} \\  0_{2,n-1}  & 0_{n-1,n-1} 
\end{pmatrix} . 
$$
We have $Z^*_{\pm}=Z_{\pm}$, $Z^2_{\pm}=Z_{\pm}$,
$\Tr Z_{\pm}=1$, i.e. $Z_{\pm}\in \F P^n$, and 
$\langle Z_+,Z_-\rangle =0$, i.e. $Z_+$ and  $Z_-$ are antipodal points. 
Using \eqref{eq5.24}, we can write
$$
\tau (\Pi_1,\Pi_2)= \tau (Z(v),  Z(-v)) =\sin 2v,
$$
and the equality \eqref{eq6.11} takes the form
\begin{equation}
Z(v)-Z(-v)=\tau (\Pi_1,\Pi_2) (Z_+-Z_-).
\label{eq6.12}
\end{equation}
Substituting \eqref{eq6.12} into \eqref{eq6.10}, we find that   
\begin{equation}
\tau (Z(v),\Pi)^2 -\tau (Z(-v),\Pi)^2
= \tau (\Pi_1,\Pi_2)\langle Z_+-Z_-,\Pi \rangle  .
\label{eq6.13}
\end{equation}
Substituting \eqref{eq6.13} into \eqref{eq6.9} and using \eqref{eq6.8}, we obtain 
\begin{equation}
\theta^{\Delta}(\eta^{\natural}, \Pi_1,\Pi_2)= \tau
(\Pi_1,\Pi_2) \theta^{\Delta} (\eta^{\natural}, Z_+, Z_-),
\label{eq6.14}
\end{equation}
where 
\begin{equation}
\theta^{\Delta}(\eta^{\natural}, Z_+,Z_-)=
\int\limits_{\F P^n}  |\langle Z_+-Z_-,\Pi\rangle | \, d\mu (\Pi).
\label{eq6.15}
\end{equation}
The integral \eqref{eq6.15} is independent of $\Pi_1,\Pi_2$, 
This proves the equality \eqref{eq2.8} for  $\F P^n$ with the constant
$
\gamma (\F P^n) = \left( \int\nolimits_{\F P^n} | \langle Z_+-Z_-,
\Pi \rangle | \, d\mu (\Pi) \right)^{-1} 
$.
Notice that in this formula any pair of antipodal points in $\F P^n$ can be taken 
instead of $Z_+,Z_-$.
The proof of Theorem~\ref{thm2.1} is complete. 
\end{proof}

\section{Proof of Lemma 2.1}\label{sec7}

$\quad \quad $(i) In \eqref{eq1.23} we put $y_1=y_2=y$ to obtain  
\begin{equation}
\la_r(y,y)=v_r-v^2_r =v_rv'_r.
\label{eq7.1}
\end{equation}
Applying the Cauchy--Schwarz inequality to \eqref{eq1.6}, we obtain 
\begin{equation}
|\la_r(y_1,y_2)|\le (\la_r(y_1,y_2) \la_r(y_2,y_2))^{1/2}=
v_rv'_r.
\label{eq7.2}
\end{equation}
Using the weak invariance principle \eqref{eq1.29}, the formula \eqref{eq1.26} and 
the bound \eqref{eq7.2}, we obtain
\begin{equation}
\theta^{\Delta}_r(y_1,y_2)\le 2v_rv'_r.
\label{eq7.3}
\end{equation}
For $r\in [0,\pi]$, we have
\begin{equation}
\sin\frac12 r\simeq r, \quad \cos\frac12 r \simeq \pi -r.
\label{eq7.4}
\end{equation}
Substituting the bounds \eqref{eq2.3} for the volumes $v_r$ and $v'_r$ 
into \eqref{eq7.2} and \eqref{eq7.3} and using \eqref{eq7.4},
we obtain \eqref{eq2.14}. Integrating \eqref{eq2.14} with $\eta\in W (d+1,d_0+1)$, 
we obtain \eqref{eq2.15}. 

(ii) We can assume that $0<r<\pi$, since  $\theta^{\Delta}_r(y_1,y_2)=0$ 
identically, if $r=0$ or  $r=\pi$.  For brevity, we write $\delta =\theta (y_1,y_2)/2$. 
The parameters $r$ and  $\delta$ vary in the region $0<r< \pi$, $0\le \delta \le \pi/2$. 
This rectangular region can be represented as a disjoint union of three triangular
regions: 

(a) $0<r<\delta$, $0\le \delta \le  \pi/2$,

(b) $\pi >r\ge \pi -\delta$, $0\le \delta \le \pi/2$,

(c) $r>\delta$, $0<r<\pi-\delta$, $0\le \delta <  \pi/2$.

In each of these regions we shall prove the bound \eqref{eq2.16}.
Notice that for $r\in [0,\pi]$, the function $\sin r/2$ is increasing while 
$\cos r/2$ is decreasing. 

{\it Case (a)}. Using the relations \eqref{eq1.25}, \eqref{eq2.2}, \eqref{eq2.3} 
and \eqref{eq7.4}, we obtain 
\begin{align}
\theta^{\Delta}_r (y_1,y_2)  \le v_r   & \simeq \int\limits^{r}_{0}
(\sin\frac12 u)^{d-1} (\cos\frac12 u)^{d_0-1}\, du  \notag
\\
& \lesssim \int\limits^{r}_{0} (\sin\frac12 u)^{d-1}\,du
\simeq (\sin\frac12 r)^{d-1}r  \notag
\\
& \lesssim  (\sin\frac12 r)^{d-1} (\cos\frac12 r)^{d_0-1}\delta.
\label{eq7.5}
\end{align}

{\it Case (b)}. Similarly, from \eqref{eq1.25}, \eqref{eq2.2}, \eqref{eq2.3} 
and \eqref{eq7.4}, we obtain  
\begin{align}
\theta^{\Delta}_r (y_1,y_2)  \le v'_r  & \simeq \int\limits^{\pi}_{r}
(\sin\frac12 u)^{d-1} (\cos\frac12 u)^{d_0-1}\, du  \notag
\\
& \lesssim \int\limits^{\pi}_{r} (\cos\frac12 u)^{d_0-1}\,du
\simeq (\cos\frac12 r)^{d_0-1}  (\pi -r)  \notag
\\
& \lesssim  (\sin\frac12 r)^{d-1} (\cos\frac12 r)^{d_0-1} \delta
\label{eq7.6}
\end{align}

{\it Case (c)}. Since $0<\theta (y_1,y_2)<\pi$, there exists the unique geodesic 
$\gamma \subset Q(d,d_0)$ 
of shortest length $\theta (y_1,y_2)$ joining points $y_1,y_2$, see \cite[Chap. VII, Sec.~10]{21}. 
Let $y_0$ denote its midpoint, i.e. $y_0\in \gamma$, $\theta(y_1,y_0)=\theta (y_2,y_0)=\delta$. 
The triangle inequality for the metric $\theta$ implies that the ball $B_{r-\delta}(y_0)$  
is contained in the intersection $B_r(y_1)\cap B_r(y_2)$. Hence 
\begin{equation}
\mu(B_r(y_1)\cap B_r(y_2))\ge v_{r-\delta}.
\label{eq7.7}
\end{equation}

Using again the relations \eqref{eq1.25}, \eqref{eq2.2}, \eqref{eq2.3} 
together with \eqref{eq7.7}, we obtain 
\begin{align}
&\theta^{\Delta}_r (y_1,y_2)\le v_r-v_{r-\delta} \simeq 
\int\limits^{r}_{r-\delta} (\sin\frac12u)^{d-1} 
(\cos\frac12u)^{d_0-1}\,  du   \notag
\\
& \lesssim
(\sin \frac12 r)^{d-1}(\cos\frac12(r-\delta))^{d_0-1}\simeq
(\sin\frac12r)^{d-1}(\pi-r+\delta)^{d_0-1}  \notag
\\
& \simeq (\sin\frac12 r)^{d-1}(\pi-r)^{d_0-1} 
\left( 1+\frac{\delta}{\pi -r}\right)^{d_0-1} \lesssim
(\sin\frac12 r)^{d-1}(\pi-r)^{d_0-1}\delta  \notag
\\
& \simeq (\sin\frac12 r)^{d-1} (\cos \frac12 r)^{d_0-1} \delta. 
\label{eq7.8}
\end{align}

Now, the bound \eqref{eq2.16} follows from the bounds 
\eqref{eq7.6} -- \eqref{eq7.8}. 
Integrating \eqref{eq2.16} with $\eta\in W(d,d_0)$, we obtain the bound \eqref{eq2.17}. 
The proof of Lemma~\ref{lem2.1} is complete.

\section*{III. Spherical functions and bounds for discrepancies and sums of distances}

\section{Preliminaries: Commutative spaces and spherical functions}\label{sec8}

In this section we outline  general facts on harmonic analysis 
on the two-point homogeneous spaces $Q(d,d_0)$.
The spaces $Q(d,d_0)$ belong to a specific and very important class of
commutative spaces. The general theory of commutative spaces can be found 
in \cite{37}, see also \cite{22} and \cite[vol.~III, Chap. 17]{35}.
For compact groups this theory is rather simple. We outline the necessary
facts in the form convenient in the subsequent calculations.

Let $G$ be a compact group and $K\subset G$ a closed subgroup.
Denote by  $\mu_G$ and $\mu_K$ Haar measures on the groups 
$G$ and $K$, correspondingly, $\mu_G(G)=\mu_K(K)=1$. As before,
$\mu$ denotes the invariant measure on the homogeneous space
$Q=G/K$, and $\mu_G=\mu_K \times \mu$.
We write $L_q(G)$, $q=1,2$, for the space of functions on $G$ 
integrable with the power $q$ with respect to the Haar measure,
$L_q(G/K)$ and $L_q(K\setminus G/K)$ for the subspaces of functions
in $L_q(G)$ satisfying  $f(gk)=f(g)$, $k\in K$, and, correspondingly, 
$f(k_1gk_2)=f(g),k_1,k_2\in K$. Obviously, functions in
these subspaces can be thought of as functions on $Q=G/K$.
The spaces $L_1(K\setminus G/K)\subset L_1(G/K)\subset L_1(G)$ are associative 
Banach algebras with the convolution product
\begin{equation}
f_1*f_2(g) =\int\limits_G f_1(gh^{-1})f_2(h)\,d \mu_G(h).
\label{eq8.1}
\end{equation}
If the algebra $L_1(K\setminus G/K)$ is commutative, the pair of groups $K\subset G$
is called a \emph{Gelfand pair}  and the corresponding homogeneous space $Q=G/K$ 
is called a \emph{commutative space}, see \cite{37}. 
Two large classes of commutative spaces are Riemannian symmetric spaces 
and two-point homogeneous spaces, see \cite{37, 22}. 
The spaces $Q(d,d_0)$  belong to both of these classes.

Consider the following unitary representation of a group $G$  
in the space $L_2(G/K)$
\begin{equation}
T(g)f(h)=f(g^{-1}h), \quad f\in L_2(G/K),\quad g,h\in G.
\label{eq8.2}
\end{equation}
and its decomposition into the orthogonal sum  
\begin{equation}
T=\widehat{\bigoplus\limits_{l\ge 0}} T_l, \quad   \quad 
L_2(G/K) =  \widehat{\bigoplus\limits_{l\ge 0}} V_l
\label{eq8.3}
\end{equation}
of unitary irreducible representations $T_l$ in finite-dimensional spaces
$V_l$. Let $m_l=\dim V_l$, and $(\cdot ,\cdot)$ denote the inner product in $V_l$.

If $Q=G/K$ is a commutative space, then the irreducible representations $T_l$ 
occurring in (8.3) are pair-wise nonequivalent and each  subspace $V_l$ in \eqref{eq8.3} 
contains a single $K$-invariant unit vector $\mathbf e^{(l)}$, i.e. 
$T_l(k)\mathbf e^{(l)}=\mathbf e^{(l)}$ for all $k\in K$. 

Fix an orthonormal basis $\mathbf e_1,\dots,\mathbf e_{m_l}$ in the space $V_l$,
such that $\mathbf e_1=\mathbf e^{(l)}$ and define the matrix elements
$t^{(l)}_{ij}(g)=(T_l(g)\mathbf e_i,\mathbf e_j)$. Then, we have
\begin{equation}
t^{(l)}_{ij}(g_1g_2)=  \sum\limits^{m_l}_{p=1} t^{(l)}_{ip} (g_1)
t^{(l)}_{pj}(g_2) \,\, \mbox{and} \,\,
t^{(l)}_{ij}(g^{-1}) = \overline{t^{(l)}_{ji}(g)}. 
\label{eq8.4}
\end{equation}
We also have the orthogonality relations
\begin{equation}
\int\limits_G t^{(l)}_{ij} (g) \overline{t^{(l')}_{ij}(g)}\, 
d\mu_G (g) = m^{-1}_{l} \delta_{ll'}\delta_{ii'}\delta_{jj'},
\label{eq8.5}
\end{equation}
where $\delta_{ab}$ is Kronecker's symbol.
The sets of functions  $\{ m^{1/2}_{l}t^{(l)}_{1j}(g)$, $j=1,\dots,m_l$, $l\ge 
0\}$ and $\{m^{1/2}_{l}t^{(l)}_{11}(g), l\ge 0\}$ are orthonormal bases in
the spaces $L_2(G/K)$ and $L_2(K\setminus G/K)$, correspondingly, 
see \cite[vol.~I, Sec.~2.3]{35} ( notice that in \cite{35} the subgroup $K$ in a Gelfand
pair $K\subset G$ is called \emph{massive} ).

The matrix elements $\varphi_l(g)=t^{(l)}_{11} (g)$  are called  
\emph{zonal spherical functions} or simply \emph{spherical functions} 
(the matrix elements $t^{(l)}_{1j}(g)$, $j=2,\dots,m_l$
are called associated spherical functions).
The definition of $\varphi_l(g)$ and the formula \eqref{eq8.4} immediately imply  
that all spherical functions are continuous, $\varphi_l(e)=1$, 
where $e$ is the unit element in $G$,  
$| \varphi_l(g)|\le 1$ for all $g\in G$,  
\begin{equation}
\varphi_l(g_1g^{-1}_{2}) = \sum\limits^{m_l}_{j=1} 
t^{(l)}_{1j}(g_1)
\overline{t^{(l)}_{1j}(g_2)}, \,\, \mbox{and} \,\,
\varphi_l(g)  = \overline{\varphi_l(g^{-1})}. 
\label{eq8.6}
\end{equation}
It follows from \eqref{eq8.6} that $\varphi_l$ is positive definite: 
\begin{equation}
\sum\limits_{1\le i,j\le N} \overline{c_i}  c_j \varphi_l (g^{-1}_{i} 
g_j)\ge 0
\label{eq8.7}
\end{equation}
for any $g_1,\dots, g_N\in G$ and any complex numbers
$c_1,\dots, c_N$.

From \eqref{eq8.1}, \eqref{eq8.5} and \eqref{eq8.6}, we obtain 
the following relation for the convolution of two spherical functions
\begin{equation}
(\varphi_l* \varphi_{l'})(g) =\delta_{ll'}m^{-1}_{l} \varphi_l(g).
\label{eq8.8}
\end{equation}
Putting $g=e$ in \eqref{eq8.8}, we obtain the following formula for the dimensions $m_l$
of irreducible representations in \eqref{eq8.3}
\begin{equation}
m_l = \left(\int\limits_G |\varphi_l (g)|^2 \, d\mu_G (g) \right)^{-1}. 
\label{eq8.9}
\end{equation}

Functions $f\in L_2(K\setminus G/K)$ have the following expansions
\begin{equation}
f(g)\sim \sum\limits_{l\ge 0} m_lc_l (f) \varphi_l(g),
\label{eq8.10}
\end{equation}
where $\sim$ denotes the $L_2$-convergence,  Fourier coefficients are given by 
\begin{equation}
c_l(f)=\int\limits_G f(g) \overline{\varphi_l (g)} \, d\mu_G(g),
\label{eq8.11}
\end{equation}
and Parseval's equality has the form
$
\int\nolimits_G |f(g)|^2 \, d\mu_{G}(g)= \sum\nolimits_{l\ge 0} m_l\,|c_l(f)|^2.
$
Actually, this is the Peter--Weyl theorem written for the 
space $L_2(K\setminus G/K)$, see \cite[vol.~I, Chap. 2]{35}

Substituting the expansion \eqref{eq8.9} for two functions $f_1,f_2\in L_2(K\setminus G/K)$ 
into \eqref{eq8.1} and using the relation \eqref{eq8.8}, we obtain 
\begin{equation}
f_1* f_2(g) = \sum\limits_{l\ge 0} m_l\, c_l(f_1)\,c_l(f_2)\, \varphi_l(g).
\label{eq8.13}
\end{equation}
Applying the Cauchy--Schwarz inequality to \eqref{eq8.13}, 
we observe
that the series \eqref{eq8.13} converges absolutely. Since the spherical functions
$\varphi_l$ are continuous and  $|\varphi_l(g)|\le 1$, we conclude that
the convolution $f_1*f_2$ is a continuous function.

The facts listed above are true for all compact commutative spaces.
Now we wish to specify these facts for two-point 
homogeneous spaces.

Let $K\subset G$ be compact groups and $Q=G/K$  a  two-point homogeneous space with 
a $G$-invariant metric $\theta$. Suppose that $K$ is the stabilizer of 
a fixed point $y_0\in Q$.  
It follows from the definition, see section 2, that the subgroup $K$ is transitive
on each sphere  $\Sigma_r(y_0)=\{y:\theta(y,y_0)=r\}\subset Q$,
$r\in \R$.
Thus, any function $f\in L_q(K\setminus G/K)$, as a function on $Q$,
is constant on each sphere $\Sigma_r(y_0)$, and we can write 
\begin{equation}
f(g)=F(\theta(g y_0,y_0))
\label{eq8.14}
\end{equation}
with a function $F(r)$, $r\in \R$. In other words, the set of double cosets 
$K\setminus G/K$  
is in one-to-one correspondence with the set of radii $\R$.

Using \eqref{eq8.14}, the convolution \eqref{eq8.1} can be written in the form
\begin{align}
(f_1*f_2)(g^{-1}_{1}g_2) & =\int\limits_G F_1(\theta(g_1y_0,gy_0))
F_2 (\theta (g y_0,g_2y_0)) \, d\mu (g)
\notag
\\
&  = \int\limits_Q F_1(\theta(y_1,y)) F_2(\theta(y,y_2)) \, d\mu (y),
\label{eq8.15}
\end{align}
where $y_1=g_1y_0$, $y_2=g_2y_0$. 

For a function of the form \eqref{eq8.14} we have
\begin{equation}
\int\limits_G |f(g)|^2\,d\mu_G(g) =
\int\limits_Q |F(\theta(y,y_0))|^2 \, d\mu (y) =
\int\limits_{\R} |F(r)|^2 \, dv_r,
\label{eq8.16}
\end{equation}
where the last integral is thought of as a Stieltjes integral with
the nondecreasing function $v_r=\mu(B_r(y_0))$, $r\in \R$.
It follows from \eqref{eq8.14} and \eqref{eq8.16} that the mapping $f\to F$ 
is an isometry 
of the space $L_2(K\setminus G/K)$ onto the space $L_2(\R,v_r)$
of functions $F(r)$, $r\in \R$, with the norm 
$
\| F\|= (\int\nolimits_{\R} |F(r)|^2 \, dv_r)^{1/2}.
$

Since the spherical functions $\varphi_l\in L_2(K\setminus G/K)$, 
they can be written in the form \eqref{eq8.14}:
\begin{equation}
\varphi_l(g)=\Phi_l(\theta(gy_0,y_0)),
\label{eq8.17}
\end{equation}
where $\Phi_l\in L_2(\R,v_r)$, and putting $y_1=g_1y_0$, $y_2=g_2y_2$, $g_1,g_2\in G$,
we can write 
\begin{equation}
\varphi_l(g^{-1}_{1}g_2) =\Phi_l (\theta(g_1y_0, g_2y_0)) =\Phi_l
(\theta (y_1,y_2)).
\label{eq8.18}
\end{equation}
It follows from the properties of $\varphi_l$ that $\Phi_l$ are continuous and real-valued, 
$\Phi_l(0)=1$, $|\Phi_l(r)|\le 1$, $r\in \R$.
The set of functions $\{ m^{1/2}_{l} \Phi_l, l\ge 0\}$ is an orthonormal basis
in the space $L_2(\R,v_r)$ and the expansion \eqref{eq8.10} for  $F\in L_2(\R,v_r)$ 
takes the form 
\begin{equation}
F(r)\sim \sum\limits_{l\ge 0}  m_l c_l(F) \,  \Phi_l(r)
\label{eq8.19}
\end{equation}
with the Fourier coefficients  
\begin{equation}
c_l(F)=\int\limits_{\R} F(r) \Phi_l(r) \, dv_r
\label{eq8.20}
\end{equation}
and Parseval's equality
$
\int\nolimits_{\R}|F(r)|^2 \, dv_r=\sum\nolimits_{l \ge 0} m_l \, |c_l (F) |^2.
$

Comparing the relations  \eqref{eq8.13} and \eqref{eq8.15},  we arrive at 
the following formula
\begin{equation}
\int\limits_Q F_1(\theta(y_1,y)) F_2(\theta(y,y_2)) \, d\mu (y)
= \sum\limits_{l\ge 0} m_lc_l(F_1)c_l(F_2)\Phi_l(\theta (y_1,y_2)).  
\label{eq8.21}
\end{equation}

For the spaces $Q=Q(d,d_0)$ the matrix elements $t^{(l)}_{1j} (g)$ are 
eigenfunctions of the Laplace--Beltrami operator on $Q$ and the
spherical functions $\varphi_l(g)=t^{(l)}_{11} (g)$ are eigenfunctions
of the radial part of this operator and can be found explicitly,
see 
\cite[p.~178]{19}, 
\cite[Chap.~V, Theorem.~4.5]{22},
\cite[pp.~514--512, 543--544]{24}, \cite[Theorem.~11.4.21]{37}.
For the functions  $\Phi_l$ in \eqref{eq8.17}, we have
\begin{equation}
\Phi_l(r) =\Phi^{(\alpha,\beta)}_l(r) = 
\frac{P^{(\alpha,\beta)}_l(\cos r)}{P^{(\alpha,\beta)}_l(1)}
,\quad r\in \R = [0, \pi ],
\label{eq8.22}
\end{equation}
where $P^{(\alpha,\beta)}_l(z)$ are the standard Jacobi polynomials of degree $l$ 
normalized by 
\begin{equation}
P^{(\alpha,\beta)}_l(1) =
\begin{pmatrix} \alpha +l \\ l  \end{pmatrix}  =
\frac{(\alpha +1)\dots(\alpha+l)}{l!} \simeq l^{\alpha},
\label{eq8.23}
\end{equation}
see \cite{34}.  The parameters $\alpha, \beta$ in \eqref{eq8.23} and
the dimensions $d$, $d_0$ in $Q(d,d_0)$ are related by
\begin{equation}
\alpha =\frac12 d - 1,\quad \beta  =\frac12 d_0  -1
\label{eq8.24}
\end{equation}
In what follows, we use the parameters $\alpha,\beta$ along with 
the dimensions $d$, $d_0$, 
assuming they are related by \eqref{eq8.24}. With this assumption we have 
$\alpha \ge \beta \ge - 1/2$ always, since $d$ and $ d_0 \ge 1$. 
Notice that $|P^{(\alpha,\beta)}_l (z)| \le P^{(\alpha,\beta)}_l (1)$
for $z\in [-1,1]$ and $\alpha \ge \beta \ge - 1/2$.

We have the following orthogonality relations for Jacobi polynomials,
see \cite[Eq.~(4.3.3)]{34},
\begin{align}
& \int\limits^{\pi}_{0}P^{(\alpha,\beta)}_l (\cos u)
P^{(\alpha,\beta)}_{l'}(\cos u) (\sin \frac12 u)^{d-1} 
(\cos \frac 12 u)^{d_0-1}\, du
\notag
\\
& = ( \frac12 )^{\alpha+\beta+1} \int\limits^1_{-1}
P^{(\alpha,\beta)}_l(z) P^{(\alpha,\beta)}_{l'} (z)
(1-z)^{\alpha}(1+z)^{\beta} \, dz =M^{-1}_l \delta_{ll'}, 
\label{eq8.25}
\end{align}                                               
where $M_0=\kappa(d,d_0)$ and 
\begin{equation}
M_l=(2l+\alpha+\beta+1) 
\frac{\Gamma (l+1)\Gamma (l+\alpha+\beta+1)}{\Gamma(l+\alpha+1)
\Gamma (l+\beta+1)} \simeq l, \; \; l\ge 1.
\label{eq8.26}
\end{equation}
Substituting the expressions for spherical functions \eqref{eq8.17}, \eqref{eq8.22}  
into 
the formula \eqref{eq8.9} and using \eqref{eq8.25}, we obtain the following explicit formula
for the dimensions $m_l$ of irreducible representations in \eqref{eq8.3}:
\begin{equation}
m_l=M_l\,B(d/2,d_0/2) 
\begin{pmatrix} \alpha +l \\ l \end{pmatrix}^2 \simeq l^{d-1} .
\label{eq8.27}
\end{equation}

For functions $F\in L_2([0,\pi],v_r)$ the expansion \eqref{eq8.19} takes the form 
\begin{equation}
F(r) \sim \sum\limits_{l\ge 0} M_l\, C_l(F)\, P^{(\alpha,\beta)}_l (\cos r),
\label{eq8.28}
\end{equation}
with the Fourier-Jacobi coefficients  
\begin{equation}
C_l(F)=\int\limits^{\pi}_{0} F(u) P^{(\alpha,\beta)}_l  (\cos u)\,
(\sin \frac12u)^{d-1}\, (\cos \frac12u)^{d_0-1}\,du. 
\label{eq8.29}
\end{equation}
and Parseval's equality 
$
\int\nolimits_{\R}|F(r)|^2 \, dv_r=\kappa(d,d_0) \sum\nolimits_{l\ge 0} M_l \,|C_l(F)|^2.
$
The Fourier--Jacobi coefficients \eqref{eq8.29} and Fourier coefficients \eqref{eq8.20}
are related by
\begin{equation}
c_l(F)=C_l(F)\, \frac{\kappa(d,d_0)}{P^{(\alpha,\beta)}_l(1)}, \quad l\ge 0.
\label{eq8.30}
\end{equation}

Using the relations \eqref{eq8.22} and \eqref{eq8.30}, 
we can write the formula \eqref{eq8.21} in the form   
\begin{align}
&\int\limits_Q F_1(\theta (y_1,y))\, F_2 (\theta(y,y_2)) \, d\mu (y)
\notag
\\
& =  \kappa(d,d_0)\, \sum\limits_{l\ge 0} M_l\, C_l (F_1)\, C_l (F_2)\,
\frac{P^{(\alpha,\beta)}_l(\cos \theta (y_1,y_2))}{P^{(\alpha,\beta)}_l(1)}.
\label{eq8.31}
\end{align}
This formula will be used in the next section to obtain spherical function expansions 
for discrepancies and metrics .

The condition of positive definiteness \eqref{eq8.7} for the spherical 
functions \eqref{eq8.17}, \eqref{eq8.23} will be used in section~\ref{sec11} 
in the following special form
\begin{equation}
\varphi_l[\D_N] =\sum\limits_{x_1,x_2\in \D_N} 
\frac{P^{(\alpha,\beta)}_l(\cos \theta 
(x_1,x_2))}{P^{(\alpha,\beta)}_l(1)} \ge 0,
\label{eq8.32}
\end{equation}
where $\D_N\subset Q (d,d_0)$ is an arbitrary $N$-point subset.
Obviously, the conditions \eqref{eq3.2}, \eqref{eq3.3} in the definition 
of $t$-designs $\D_N\subset Q(d,d_0)$
are equivalent to the following equalities 
\begin{equation}
\varphi_l[\D_N] =0, \quad l=0,1,\dots,t,
\label{eq8.33}
\end{equation}
see also \cite{4, 24}.
The relations \eqref{eq8.33} can 
be used as an alternative to the definition of $t$-designs given 
before in section 3, see \cite{4, 24}.

\section{Spherical function expansions for discrepancies and metrics}\label{sec9}

In this section we obtain explicit spherical function expansions 
for the kernels \eqref{eq1.6}, \eqref{eq1.8} and the symmetric difference 
metrics \eqref{eq1.13}, \eqref{eq1.14} 
on the spaces $Q(d,d_0)$. The coefficients of these expansions will be estimated 
in the next section. 

First of all, we recall the main facts on Jacobi polynomials 
$P^{(\alpha,\beta)}_l(z)$, $z\in [-1,1]$, $\alpha \ge -1/2$,
$\beta \ge -1/2$, as $l\to \infty$. It is known, see \cite{34}, that 
Jacobi polynomials are behaved extremely irregularly on 
the interval $z\in [-1,1]$: inside the interval they oscillate and are of 
order $l^{-1/2}$, while in neighborhoods of the end points $z=1$ and $z=-1$  
they increase rapidly up to the quantities of order $l^{\alpha}$ and $l^{\beta}$, 
correspondingly. It is convenient to introduce the following function 
to describe such a behavior of Jacobi polynomials:
\begin{equation}
J^{(\alpha,\beta)}_l(r)= (\sin \frac12 r)^{\alpha+\frac12} (\cos \frac12 r)^{\beta +\frac12}
P^{(\alpha,\beta)}_l(\cos r), \quad r\in [0,\pi].
\label{eq9.1}
\end{equation}
We have the following two bounds.

(a) In the interval $r\in [c_0l^{-1},\pi-c_0l^{-1}]$, where $c_0>0$ is an arbitrary constant, 
we have the asymptotic formula 
\begin{equation}
J^{(\alpha,\beta)}_l (r) = (\pi l)^{-1/2}
\{ \cos [(l+l_0)r+r_0]+O((l\sin r)^{-1})\},
\label{eq9.2}                            
\end{equation}
where $l_0=(\alpha+\beta+1)/2$, $r_0=-\pi (2\alpha+1)/4$, see \cite[Theorem.~8.21.3]{34}.

(b) In the intervals $r\in [0,c_0l^{-1}]$ and $r\in [\pi -c_0l^{-1}, \pi]$, we have the bound
$J^{(\alpha,\beta)}_l(r)=O(l^{-1/2})$, see \cite[Theorem.~7.32.2]{34}.  This bound together 
with \eqref{eq9.2} implies  the following bound
\begin{equation}
|J^{(\alpha,\beta)}_l(r)|  <c(l+1)^{-1/2}, \quad l\ge 0,
\label{eq9.3}
\end{equation}
which holds uniformly for all $r\in [0,\pi]$ with the constant $c$ 
depending only on $\alpha, \beta$ and $c_0$. 

Consider the measure of the intersection of two balls
$B_r(y_1)$ and $B_r(y_2)$ in the space $Q=Q(d,d_0)$
\begin{equation}
\mu_r(y_1,y_2)=\mu(B_r(y_1)\cap B_r(y_2))
=  \int\limits_Q \chi_r(\theta (y_1,y))\chi_r(\theta (y,y_2)) \,
d \mu (y),
\label{eq9.4}
\end{equation}     
where $\chi_r(\cdot)$ is the characteristic function of the interval
$[0,r]$,  $0\le r\le \pi$, see \eqref{eq1.16}. 

\begin{lemma}\label{lem9.1} The kernel \eqref{eq9.4} has the following 
spherical function expansion 
\begin{equation}
\mu_r(y_1,y_2)
= v^2_r+\kappa(d,d_0) \sum\limits_{l\ge 1} l^{-2} M_la_l(r)
\frac{P^{(\alpha,\beta)}(\cos\theta (y_1,y_2))}{P^{(\alpha,\beta)}_l(1)}, 
\label{eq9.5}
\end{equation}
where $v_r=\mu (B_r (y))$  and 
\begin{align}
a_l(r)& =(\sin\frac 12 r)^{2d}(\cos \frac12 r)^{2d_0}
\left\{ P^{(\alpha+1,\beta+1)}_{l-1} (\cos r)\right\}^2
\notag
\\
& = (\sin \frac12 r)^{d-1} (\cos \frac 12 r)^{d_0-1}
\left\{J^{(\alpha+1,\beta+1)}_{l-1}  (r) \right\}^2.
\label{eq9.6}
\end{align}
The coefficients in \eqref{eq9.5}  satisfy
\begin{equation}
M_la_l(r)\le c (\sin \frac 12 r)^{d-1} (\cos \frac12 r)^{d_0-1}
\label{eq9.7}
\end{equation}
with a constant depending only on $d$ and $d_0$. 
Furthermore, we have the equality
\begin{equation}
\kappa(d,d_0)\sum\limits_{l\ge 1} l^{-2} M_la_l(r) = v_r-v^2_r=v_rv'_r.
\label{eq9.8}
\end{equation}
\end{lemma}

\begin{proof}[Proof] Applying the expansion \eqref{eq8.31} to the 
integral \eqref{eq9.4}, we obtain 
\begin{equation}
\mu_r(y_1,y_2)=\kappa(d,d_0) \sum\limits_{l\ge 0}M_l \{C_l(\chi_r)\}^2
\:\frac{P^{(\alpha,\beta)}_l(\cos\theta (y_1,y_2))}{P^{(\alpha,\beta)}_l(1)}, 
\label{eq9.9}
\end{equation}
where $C_l(\chi_l)$ are Fourier-Jacobi coefficients \eqref{eq8.29} of  
the characteristic function $\chi_r$. We have 
\begin{align}
C_l(\chi_r) & = \int\limits^{r}_{0} P^{(\alpha,\beta)}_l(\cos u) (\sin \frac12 u)^{d-1}
(\cos\frac 12 u)^{d_0-1} \, du 
\notag
\\
& = (\frac12)^{\frac{d-1}{2}+\frac{d_0-1}{2}}
\int\limits^{1}_{\cos r} (1-z)^{\alpha} (a+z)^{\beta}
P^{(\alpha,\beta)}_l (z) \, dz.
\label{eq9.10}
\end{align}
In view of \eqref{eq2.2}, we have  $C_0(\chi_r)=\kappa(d,d_0)^{-1}v_r$. For $l\ge 1$ 
we use Rodrigues' formula for Jacobi polynomials, see~\cite[Eq.~(4.3.1)]{34},
\begin{equation}
P^{(\alpha,\beta)}_l(z) =
\frac{(-1)^l}{2^ll!} (1-z)^{-\alpha} (1+z)^{-\beta}
\frac{d^l}{dz^l} 
\left\{ (1-z)^{l+\alpha} (1+z)^{l+\beta} \right\}.
\label{eq9.11}
\end{equation}
Substituting \eqref{eq9.11} into \eqref{eq9.10}, we obtain
\begin{align*}
& \int\limits_{\cos r}^{1} (1-z)^{\alpha} (1+z)^{\beta}
P^{(\alpha,\beta)}_l (z) \, dz
\\
& = (2l)^{-1} (1-\cos r)^{\alpha +1} (1+\cos r)^{\beta +1}
P^{(\alpha+1,\beta+1)}_{l-1} (\cos r)
\\
& = 2^{\alpha +\beta +1} l^{-1} (\sin \frac12r)^{2\alpha +2}
(\cos \frac12r)^{2\beta+2}
P^{(\alpha+1,\beta+1)}_{l-1} (\cos r).            
\end{align*}
In view of the definitions \eqref{eq9.1} and \eqref{eq8.24}, we have
\begin{align}
C_l (\chi_r) &=l^{-1} (\sin \frac12 r)^d (\cos \frac 12 r)^{d_0}
P^{(\alpha+1,\beta+1)}_{l-1}(\cos r) 
\notag
\\
& = l^{-1} (\sin \frac12 r)^{\frac{d-1}{2}}
(\cos \frac12 r)^{\frac{d_0-1}{2}}
J^{(\alpha+1,\beta+1)}_{l-1} (r). 
\label{eq9.12}
\end{align}
Substituting \eqref{eq9.12} into \eqref{eq9.9}, we obtain 
the formulas \eqref{eq9.5} and \eqref{eq9.6}.

The bound \eqref{eq9.7} follows from \eqref{eq9.6}, since $M_l\simeq l$, 
see~\eqref{eq8.26}, and 
$J^{(\alpha+1,\beta+1)}_{l-1}(r)\lesssim l^{-1/2}$, see \eqref{eq9.3}.  

From \eqref{eq9.4}, we obtain $\mu_r(y,y)=v_r$. Putting  $y_1=y_2=y$ 
in \eqref{eq9.5}, we obtain \eqref{eq9.8}. In fact, the formula \eqref{eq9.8} 
is Parseval's equality for
the expansion \eqref{eq8.28} of the characteristic function $\chi_r$.
\end{proof}

An immediate corollary of Lemma~\ref{lem3.1} is the following.
\begin{theorem}\label{thm9.1} For any space $Q(d,d_0)$ the following
spherical function expansions hold:

{\rm (i)}  For the kernels $\la_r(y_1,y_2)$, see \eqref{eq1.6}, and the metrics
$\theta^{\Delta}_r(y_1,y_2)$, see \eqref{eq1.14}, we have 
\begin{equation}
\la_r(y_1,y_2)=\kappa(d,d_0) \sum\limits_{l\ge 1} l^{-2}M_la_l(r)
\frac{P^{(\alpha,\beta)}_{l}(\cos \theta 
(y_1,y_2))}{P^{(\alpha,\beta)}_{l}(1)},
\label{eq9.13}
\end{equation}
\begin{align}
 \theta^{\Delta}_r(y_1,y_2) & =\langle \theta^{\Delta}_r\rangle -\kappa(d,d_0)
\sum\limits_{l\ge 1} l^{-2} M_la_l(r)
\frac{P^{(\alpha,\beta)}_{l}(\cos \theta 
(y_1,y_2))}{P^{(\alpha,\beta)}_{l}(1)},
\notag
\\
& = \kappa(d,d_0)\sum\limits_{l\ge 1} l^{-2} M_la_l(r) 
\left[ 1-\frac{P^{(\alpha,\beta)}_{l}(\cos \theta 
(y_1,y_2))}{P^{(\alpha,\beta)}_{l}(1)} \right],
\label{eq9.14}
\end{align}
where $\langle \theta^{\Delta}_r\rangle= v_rv'_r$ is the average value of metric
$\theta^{\Delta}_r$, see \eqref{eq1.26}, 
and the coefficients $a_l(r)$ are defined in \eqref{eq9.6}. 

{\rm (ii)} If the weight function $\eta \in W (d,d_0)$, then for the kernels
$\la (\eta, y_1,y_2)$, see \eqref{eq1.8}, and the metrics
$\theta^{\Delta}(\eta,y_1,y_2)$, see \eqref{eq1.13}, we have 
\begin{equation}
\la(\eta, y_1,y_2)=\kappa(d,d_0) \sum\limits_{l\ge 1} l^{-2}M_lA_l(\eta)
\frac{P^{(\alpha,\beta)}_{l}(\cos \theta 
(y_1,y_2))}{P^{(\alpha,\beta)}_{l}(1)},
\label{eq9.15}
\end{equation}
\begin{align}
\theta^{\Delta}(\eta, y_1,y_2) & =\langle \theta^{\Delta} (\eta) \rangle -\kappa(d,d_0)
\sum\limits_{l\ge 1} l^{-2} M_lA_l(\eta)
\frac{P^{(\alpha,\beta)}_{l}(\cos \theta 
(y_1,y_2))}{P^{(\alpha,\beta)}_{l}(1)},
\notag
\\
& = \kappa(d,d_0) \sum\limits_{l\ge 1} l^{-2} M_lA_l(\eta) 
\left[ 1-
\frac{P^{(\alpha,\beta)}_{l}(\cos \theta 
(y_1,y_2))}{P^{(\alpha,\beta)}_{l}(1)} \right],
\label{eq9.16}
\end{align}
where $\langle \theta^{\Delta}(\eta)\rangle$ is the average value of metric
$\theta^{\Delta}(\eta)$, see \eqref{eq1.18}, and the coefficients $A_l(\eta)$ 
are defined by
\begin{equation}
A_l(\eta)=\int\limits^{\pi}_{0} \eta (a) a_l(u) \, du. 
\label{eq9.17}
\end{equation}
\end{theorem}

\begin{proof}[Proof] (i) Substituting the expansion \eqref{eq9.5} into \eqref{eq1.23} 
and \eqref{eq1.25}, we obtain the expansions \eqref{eq9.13} and \eqref{eq9.14}. 
Notice that in the second equality in \eqref{eq9.14} the formula \eqref{eq9.8} 
has been taken into account. 

(ii)  In view of the bound \eqref{eq9.7}, the series \eqref{eq9.13} and 
\eqref{eq9.14} can be integrated 
term by term with  $\eta\in W(d,d_0)$. This gives the expansions 
\eqref{eq9.15} and \eqref{eq9.16}. 
\end{proof}

Notice that by Theorem~\ref{thm2.1} the chordal metric $\tau$ is 
a symmetric difference metric \eqref{eq1.13}
with the weight function  $\eta^{\natural}$ and,
therefore, it has the expansion \eqref{eq9.16}.  
At the same time, the chordal metric can be written as follows  
\begin{equation}
\tau(y_1,y_2)=c(\alpha,\beta)  \left[ 1-
\frac{P^{(\alpha,\beta)}_{1}(\cos \theta 
(y_1,y_2))}{P^{(\alpha,\beta)}_{1}(1)} \right]^{1/2},
\label{eq9.18}
\end{equation}
with the constant
$
c(\alpha,\beta)= \left( {\alpha +1}/{\alpha +\beta+2}\right)^{1/2}
= \left( {d}/{d+d_0}\right)^{1/2}.
$

Indeed, by Rodrigues' formula \eqref{eq9.11} 
$P^{(\alpha,\beta)}_{1}(z)=\frac12 
(\alpha+\beta+2)z+\frac12(\alpha-\beta)$, and 
\begin{equation}
\frac12(1-z)=\frac{\alpha +1}{\alpha+\beta+2}
\left[  1- \frac{P^{(\alpha,\beta)}_{1}(z)}{P^{(\alpha,\beta)}_{1}(1)} 
\right].
\label{eq9.20}
\end{equation}
On the other hand, by the definitions \eqref{eq2.4} and \eqref{eq5.24}  
\begin{equation}
\tau (y_1,y_2)=\sin\frac12\theta(y_1,y_2) =
\left[\frac12\left(1-\cos \theta(y_1,y_2)\right)\right]^{1/2}.
\label{eq9.21}
\end{equation}
Comparing \eqref{eq9.20} and \eqref{eq9.21}, we obtain \eqref{eq9.18}.

\section{Bounds for Fourier-Jacobi coefficients}\label{sec10}

In this section we estimate the following coefficients
\begin{align}
a_l(r) & = (\sin \frac12r)^{d-1}(\cos \frac12 r)^{d_0-1}
\left\{ J^{(\alpha+1,\beta+1)}_{l-1}(r)\right\}^2,
\label{eq10.1}
\\
A_l(\eta) & =\int\limits_{0}^{\pi} \eta (u) a_l(u) \, du,
\label{eq10.2}
\\
A_l(\chi_r) & =\int\limits^{\pi}_{0}\chi_r(u)a_l(u) \,du 
=\int\limits^{r}_{0} a_l(u)\, du, 
\label{eq10.3}
\end{align}
where $J^{(\alpha,\beta)}_l(\cdot)$ is defined in \eqref{eq9.1}. In fact, we prove special
weighted bounds for Jacobi polynomials.
\begin{lemma}\label{lem10.1} Let the weight function $\eta\in W(d,d_0)$, 
$\eta \ne 0$, then the following bounds hold: 

{\rm (i)} For $0<r\le \pi$ and $l\ge 1$, we have
\begin{equation}
A_l(\eta) >c r^{-d+1}  a_l(r). 
\label{eq10.4}
\end{equation}

{\rm (ii)} There exists a constant $L\ge 1$, depending only on $\alpha$ and $\beta$, 
such that for $0<r\le \pi/2$ and $lr >L$, we have
\begin{equation}
A_l(\eta) <C r^{-d} A_l(\chi_r).
\label{eq10.5}
\end{equation}

The positive constants $c$ and $C$ in \eqref{eq10.4} and \eqref{eq10.5} depend only 
on $\alpha$, $\beta$ and $\eta$. 
\end{lemma}

\begin{proof}[Proof] The asymptotic formula \eqref{eq9.2} implies the following relations 
\begin{align}
J^{(\alpha+1,\beta+1)}_{l-1}(r) & = (\pi l)^{-1}
\left\{\sin [(l+l_0)r+r_0] +O((l\sin r)^{-1})  \right\},
\label{eq10.6}
\\
\{J^{(\alpha+1,\beta+1)}_{l-1}(r)\}^2 & =
\left\{ \frac12 -\frac 12\cos 2 [(l+l_0)r+r_0] + R_l(r)\right\},
\label{eq10.7}
\end{align}
where the error term $R_l(r)$ satisfies  
\begin{equation}
R_l(r)=
\begin{cases}
O(l^{-1}) & \quad \text{for} \quad 0<c_0\le r\le \pi -c_0, \\
O((lr)^{-1}) & \quad \text{for} \quad l^{-1}\le r \le \pi/2,
\end{cases}               
\label{eq10.8}
\end{equation}
where $0<c_0<\pi/2$ is arbitrary fixed.

(i) Since $\eta \in W(d,d_0)$, $\eta\ne 0$, a sufficiently small constant $0<c_0<\pi/2$ 
can be chosen to satisfy
\begin{align}
& \int\limits^{\pi-c_0}_{c_0} \eta (u) (\sin \frac12 u)^{d-1}
(\cos \frac12 u)^{d_0-1} \, du
\notag 
\\
&  \ge \frac12 \int\limits^{\pi}_{0} \eta(u) (\sin \frac12 u)^{d-1}
(\cos \frac 12 u)^{d_0-1} \, du =\frac 12 \| \eta\|_{d,d_0} >0.
\label{eq10.9}
\end{align}
Using \eqref{eq10.9}, \eqref{eq10.7} and the first bound in \eqref{eq10.8}, we obtain 
\begin{align}
& A_l (\eta) \ge \int\limits^{\pi-c_0}_{c_0}
\eta(u)(\sin\frac12 u)^{d-1}(\cos\frac12 u)^{d_0-1}
\left \{J^{(\alpha +1,\beta+1)}_{l-1} (u)\right\}^2\, du  
\notag
\\
& \ge (\pi l)^{-1} \Big\{ \frac14 \| \eta\|_{d,d_0}  
\! -\! \frac 12\int\limits^{\pi-c_0}_{c_0} \eta(u) (\sin \frac12 u)^{d-1}
(\cos \frac12 u)^{d_0-1} \cos 2 
[(l\!+\!l_0)u\!+\!r_0] \, du 
\notag
\\
& \qquad+O(l^{-1})\Big\}
 = (4\pi l)^{-1}   \| \eta \|_{d,d_0} + o(1),
\label{eq10.10}
\end{align}
where in the last equality the Riemann-Lebesgue lemma has been used.
Hence
\begin{equation}
A_l(\eta) \ge (8\pi l)^{-1}  \| \eta \|_{d,d_0}
\label{eq10.11}
\end{equation}
for all sufficiently large $l>l_1$. We have
\begin{equation}
\min\limits_{1\le l\le l_1} lA_l(\eta) >0,
\label{eq10.12}
\end{equation}
since, $A_l(\eta)>0$ for all $l\ge 1$. From \eqref{eq10.11} and \eqref{eq10.12}, 
we conclude  that the bound  
\begin{equation}
A_l(\eta) > c_1l^{-1}
\label{eq10.13}
\end{equation}
holds for all $l\ge 1$  with a constant $c_1>0$ depending only on $\alpha, \beta$
and $\eta$. 

From the other hand, the bound \eqref{eq9.3} implies 
\begin{equation}
r^{-d+1}a_l(r) =r^{-d+1}
(\sin \frac12 r)^{d-1}  (\cos \frac12 r)^{d_0-1}
\left\{ J^{(\alpha+1,\beta+1)}_{l-1}(r)\right\}^2
\le c_2l^{-1}
\label{eq10.14}
\end{equation}
Comparing the bounds \eqref{eq10.13} and \eqref{eq10.14}, 
we obtain the bound \eqref{eq10.4}  with $c=c_1c^{-1}_2$. 

(ii) Let $0<r\le \pi/2$ and $lr \ge L$, where $L\ge 1$ is a constant
which will be fixed later. 
From the definition \eqref{eq10.3}, we obtain 
\begin{align}
& r^{-d} A_l(\chi_r) \ge r^{-d}\int\limits^{r}_{r/2} a_l(u) \, du
\notag
\\
& \ge r^{-d} (\sin \frac14 r)^{d-1} (\cos \frac 12 r)^{d_0-1}
\int\limits^{r}_{r/2} 
\left\{ J^{(\alpha+1,\beta +1)}_{l-1}(u)\right\}^2\, du
\notag
\\
& > c_1r^{-1} \int\limits^{r}_{r/2} 
\left\{ J^{(\alpha+1,\beta +1)}_{l-1}\right\}^2\, du ,
\label{eq10.15}
\end{align}
where one can put  $c_1=(1/8)^{d-1}(1/2)^{d_0-1}$.             
Using the asymptotic formula \eqref{eq10.7} and the second bound in \eqref{eq10.8}, 
we obtain
\begin{align}
& r^{-1}\int\limits^{r}_{r/2} 
\left\{ J^{(\alpha+1,\beta +1)}_{l-1}(u)\right\}^2\, du
\notag
\\
& =  (\pi l)^{-1} \left\{  \frac 14  -\frac 12 r^{-1} 
\int\limits^{r}_{r/2} \cos 2 [(l+l_0) u+r_0]\, du +O(L^{-1})\right\} .
\label{eq10.16}
\end{align}
The integral on the right-hand side in \eqref{eq10.16}   
is of order $O((r l)^{-1})\lesssim O (L^{-1})$. 
Substituting \eqref{eq10.16} into \eqref{eq10.15}, we obtain 
\begin{equation}
r^{-d}A_l(\chi_r)>c_1(4\pi l)^{-1} \left\{ 1+O(L^{-1})\right\}.
\label{eq10.17}
\end{equation}
Now, in view of \eqref{eq10.17}, we can fix a sufficiently large constant $L$, 
depending only on $\alpha$ and $\beta$, to satisfy
\begin{equation}
r^{-d} A_l (\chi_r)>  c_1(8\pi l)^{-1} =c_2l^{-1}. 
\label{eq10.18}
\end{equation}
From the other hand, using the bound \eqref{eq9.3} and the definitions \eqref{eq10.2} 
and \eqref{eq2.13}, we obtain 
\begin{equation}
A_l(\eta) \le C_2\| \eta \|_{d,d_0} l^{-1} =C_3l^{-1}. 
\label{eq10.19}
\end{equation}
Comparing \eqref{eq10.18} and \eqref{eq10.19}, 
we obtain the bound \eqref{eq10.5}  with $C=C_3c^{-1}_2$. 
\end{proof}

\section{Proof of Theorems 2.2 and 3.1}\label{sec11}

Theorems~\ref{thm2.2} and 3.1 are immediate corollaries of bounds
on discrepancies given below in Theorem~\ref{thm11.1}.

By Theorem 9.1 we can write the discrepancies 
\eqref{eq1.5} and \eqref{eq1.7} in the following form 
\begin{align}
\la_r[\D_N] & = \kappa(d,d_0) \sum\limits_{l\ge 1} l^{-2} M_la_l(r) 
\varphi_l[\D_N],
\label{eq11.1}
\\
\la[\eta, \D_N] & = \kappa(d,d_0) \sum\limits_{l\ge 1} l^{-2} M_lA_l(\eta) 
\varphi_l[\D_N],
\label{eq11.2}
\\
\la[\chi_r,\D_N] & = \kappa(d,d_0) \sum\limits_{l\ge 1}  l^{-2}
M_lA_l (\chi_r) \varphi_l[\D_N],
\label{eq11.3}
\end{align}
here $\D_N\subset Q (d,d_0)$ is an arbitrary $N$-point subset, and the quantities  
$\varphi_l[\D_N]\ge 0$ are defined in \eqref{eq8.32}.  
The series \eqref{eq11.1} -- \eqref{eq11.3} converge and all their terms are nonnegative. 

\begin{theorem}\label{thm11.1}  Let the weight function $\eta\in W(d,d_0)$, $\eta \ne 0$, 
then the following bounds hold:

{\rm  (i)} For any $N$-point subset $\D_N\subset Q (d,d_0)$ and an arbitrary $r, \, 0<r\le \pi$, 
we have
\begin{equation}
\la [\eta, \D_N] >c r^{-d+1} \la_r [\D_N],
\label{eq11.4}
\end{equation}

{\rm (ii)} There exists a constant $L\ge 1$, depending only $d$ and $d_0$, 
such that for any $N$-point $t$-design  $\D_N\subset Q(d,d_0)$ with $t\ge 2L/\pi$,
we have  
\begin{equation}
\la [\eta, \D_N] < C r^{-d} \la [\chi_r, \D_N], \; \; r=L t^{-1}.
\label{eq11.5}
\end{equation}

The positive constants $c$ and $C$ in \eqref{eq11.4} and \eqref{eq11.5} depend 
only on $d$, $d_0$ and $\eta$. 
\end{theorem}

\begin{proof}[Proof] (i) Applying the bound \eqref{eq10.4} to the series \eqref{eq11.1} and
\eqref{eq11.2}, we obtain the bound \eqref{eq11.4}.

(ii) If $\D_N\subset Q(d,d_0)$ is a $t$-design,  
then $\varphi [\D_N]=0$ for $l=0,1,\dots,t$, see \eqref{eq8.33}. Hence, the summation
in all series  \eqref{eq11.1} -- \eqref{eq11.3} is taken over $l>t$. 

Let $L$ be chosen as the constant indicated in Lemma~\ref{lem10.1}(ii).  
If $r=Lt^{-1}$, then we have $0<r\le \pi/2$ for $t\ge 2L/\pi$ and
$lr>L$ for $l>t$. Applying the bound \eqref{eq10.5} to the series \eqref{eq11.2} and 
\eqref{eq11.3}, we obtain the bound \eqref{eq11.5}.
\end{proof}

Now we are in position to prove Theorems~\ref{thm2.2} and 3.1. 

\begin{proof}[ Proof of Theorem~\ref{thm2.2}.] As it was explained in comments 
to Theorem~\ref{thm2.2} 
we have to prove only the left bound in \eqref{eq2.19}. From the definitions of discrepancies 
\eqref{eq1.3}, \eqref{eq1.5}, we conclude that 
$
\la_r [\D_N]\, \ge \,\, \langle N v_r \rangle^2,
$
where $\langle z \rangle=\min\{ |z-n|,n\in \bZ\}$ is the distance of  
$z\in \bR$ from the nearest integer. Define $r$ by $N v_r=1/2$, then
$\la_r[\D_N]\ge 1/2$.  In view of \eqref{eq2.3}, $r\simeq N^{-1/d}$ and the bound 
\eqref{eq11.4} implies the left bound in \eqref{eq2.19}.
\end{proof}

\begin{proof}[{ Proof of Theorem~\ref{thm3.1}.}] First of all, we notice that 
\begin{align}
&\int\limits_Q  ( \# \{ B_r(y)\cap \D_N\})^2\, d\mu (y) 
 = \int\limits_Q  \left( \sum\limits_{y_1\in \D_N} \chi 
(B_r(y),y_1)\right)^2\, d\mu (y)  \notag
\\
& = \int\limits_Q  \left( \sum\limits_{y_1\in \D_N} \chi 
(B_r(y_1),y)  \right)^2 \, d\mu (y)
= \sum\limits_{y_1,y_2\in \D_N}  \mu (B_r(y_1)\cap B_r(y_2)),
\label{eq11.6}
\end{align}
here the formula \eqref{eq1.16} has been used. Comparing  \eqref{eq11.6} 
with \eqref{eq1.24}, we obtain  
\begin{equation}
\la_r[\D_N] <\int\limits_Q
(\# \{B_r(y)\cap \D_N\} )^2 \, d\mu (y) \le
(\nu [\D_N,r])^2,
\label{eq11.7}
\end{equation}
where $\nu [\D_N, r]$ is defined in \eqref{eq3.7}. Therefore
\begin{equation}
\la [\chi_r,\D_N] =\int\limits^{r}_{0} \la_u[\D_N] \, du <r
(\nu [\D_N,r])^2,
\label{eq11.8}
\end{equation}
since $\nu [\D_N, r]$ is a nondecreasing function of $r$.
Substituting \eqref{eq11.8} into \eqref{eq11.5}, we obtain 
\begin{equation}
\la [\eta,\D_N] <Cr^{-d+1} (\nu [\D_N,r])^2.
\label{eq11.9}
\end{equation}
If  $r=Lt^{-1}$, then the bound \eqref{eq11.9} coincides with the bound \eqref{eq3.8}.
\end{proof}

\section{Additional remarks}\label{sec12}

In this section we discuss very briefly some questions related with the 
matter of the present paper. 

(i) First of all we explain 
the appearance of anomalously small errors in the formula \eqref{eq1.39}.
It is known that for the sphere $S^d$ 
the geodesic metric $\theta$ can be written as follows 
\begin{equation}
\theta(y_1,y_2)=\pi \mu(B_{\pi/2}(y_1)\Delta B_{\pi/2}(y_2)), \quad 
y_1,y_2\in S^d,
\label{eq4.12}
\end{equation}
where $B_{\pi/2}(y)=\{x\in S^d:\theta(x,y) < \pi/2\}=\{x\in S^d: (x,y) > 0\}$ is the
hemisphere centered at $y\in S^d$ and $\mu$ the standard Lebesgue measure   
normalized by (1.1), see \cite[Sec.~6.4]{17}. 
Using \eqref{eq1.14}, we can write \eqref{eq4.12} in the form 
\begin{equation}
\theta(y_1,y_2)=\pi(1-2\mu(B_{\pi/2}(y_1)\cap B_{\pi/2}(y_2))
\label{eq4.13}
\end{equation}
Notice that in this form, the equality \eqref{eq4.13} is obvious: it suffers to notice that the measure 
of the intersection of two hemispheres in \eqref{eq4.13} is a linear function of $\theta(y_1,y_2)$.
Comparing \eqref{eq4.13} and \eqref{eq1.25} and taking into account 
that $v_{\pi/2}=1/2$, we can write
\begin{equation}
\theta(y_1,y_2)=2\pi \theta^{\Delta}_{\pi/2}(y_1,y_2).
\label{eq4.14}
\end{equation}
Hence, the geodesic metric $\theta$ on the sphere $S^d$ is a symmetric difference metric.

Using the formula \eqref{eq4.14} and the invariance
principle \eqref{eq1.29} for the sphere $S^d$, we find that
$$
\theta[\D_N]=\lan \theta\ran N^2-2\pi \la_{\pi/2}[\D_N],
$$ 
where 
$$
\la_{\pi/2}[\D_N]=\int\limits_{S^d}\Lambda [B_{\pi/2}(y),\D_N]^2\,d\mu(y)
$$
where
$
\Lambda[B_{\pi/2}(y),\D_N]=\#\{B_{\pi/2}(y)\cap \D_N\}-Nv_{\pi/2}.
$
and $\lan \theta\ran =\pi/2$, see \eqref{eq1.26}.

An $N$-point subset $\D_N\subset S^d$ can be represented as a disjoint union
of two subsets 
$\D_N=\D^{(0)}_{2a}\cup \D^{(1)}_b, N=2a+b,$
where
$\D^{(0)}_{2a}=\{x\in \D_N:-x\in \D_N\}$ and $\D^{(1)}_b=\{x\in \D_N:-x\notin \D_N\}$. 
We have 
$$
\Lambda [B_{\pi/2}(y),\D_N]=\Lambda [B_{\pi/2}(y),\D^{(0)}_{2a}]+\Lambda
[B_{\pi/2}(y),\D^{(1)}_b].
$$
It is clear that $\Lambda [B_{\pi/2}(y),\D^{(0)}_{2a}]=0$ for all $y\in S^d$  
except the hyperplanes  $\lan y,x\ran =0$, $x\in \D^{(0)}_{2a}$. Hence,
$
\la_{\pi/2}[\D_N]=\la_{\pi/2}[\D^{(1)}_b].
$

Let $N=2a$ be even and $\D_N=\D^{(0)}_{2a}$, then  $\la_{\pi/2}[\D_N]=0$. 
Let $N=2a+1$ be odd and $\D_N=\D^{(0)}_{2a}\cup \D^{(1)}_1$, where $\D^{(1)}_1=\{x_0\}$  
is a one-point subset. A simple calculation shows that $\la_{\pi/2}[\{x_0\}]=\pi/2$.
Therefore, $\la_{\pi/2}[\D_N]=\pi/2$, and the relation \eqref{eq1.39} follows.

A similar proof of the relation \eqref{eq1.39} was recently given 
in \cite[Theorem 3.5]{8}. 
Additionally, these authors established the exact value $\varepsilon_N = \pi/2$ 
for odd~$N$.

The relation \eqref{eq1.39} can be also derived from the spherical function 
expansion \eqref{eq9.14} for the geodesic distance on $S^d$.
For the sphere $S^d$, we have $d_0=d$, $\beta=\alpha=d/2-1$, and 
Jacobi polynomials $P^{(\alpha,\alpha)}_{l}(z)$ coincide, up to constant factors, with  
Gegenbauer polynomials. Furthermore, $P^{(\alpha,\alpha)}_l(z)$ for even and odd $l$  
are, correspondingly, even and odd functions of $z$, see \cite[Sec.~4.7]{34}. 
Comparing the formula \eqref{eq4.14} and the expansion \eqref{eq9.14} for $r=\pi/2$, 
we obtain the following expansion for the geodesic distance on $S^d$
\begin{align}
& \theta(y_1,y_2)=   \notag
\\
&=\! 2\pi \left[ \frac14\! -\!(\frac14)^d \kappa(d,d_0)
\sum\limits_{odd \, l\ge 1} l^{-2}M_l 
\left\{ P^{(\alpha+1,\alpha+1)}_{l-1}(0) \right\}^2
\frac{P^{(\alpha,\alpha)}_{l}(\cos \theta 
(y_1,y_2))}{P^{(\alpha,\alpha)}_{l}(1)} \right]
\notag
\\
& =2\pi(\frac14)^d \kappa(d,d_0) \sum\limits_{odd \, l\ge 1} l^{-2}M_l
\left\{ P^{(\alpha+1,\alpha+1)}_{l-1}(0) \right\}^2
\left[ 1-\frac{P^{(\alpha,\alpha)}_{l}(\cos \theta 
(y_1,y_2))}{P^{(\alpha,\alpha)}_{l}(1)} \right]. 
\label{eq9.22}
\end{align}
The expansion contains spherical functions only with odd indexes. 
For odd $l$ for the sums \eqref{eq8.32}, we have  $\varphi_l[\D^{(0)}_{2a}] = 0$
and $\varphi_l[\D^{(0)}_{2a}\cup \D^{(1)}_1] = 1$, where the subsets
$\D^{(0)}_{2a}$ and $\D^{(0)}_{2a}\cup \D^{(1)}_1$ are defined as above.
Substituting these equalities into  \eqref{eq9.22}, we obtain 
the relation \eqref{eq1.39}. 

(ii) The L\'evi--Schoenberg kernel on an arbitrary metric space $\M$ with 
a metric $\rho$ is defined by 
\begin{equation}
k(\rho,y_1,y_2)=\rho (y_1,y_0) +\rho (y_2,y_0) -\rho
(y_1,y_2),
\label{eq12.5}
\end{equation}
where $y_0\in\M$ is a fixed point, see \cite{19} The metric $\rho$ can
be recovered from the kernel $k$ by 
$
\rho(y_1,y_2)= 2(k(\rho,y_1,y_1)+k(\rho,y_2,y_2)-2k (\rho,y_1,y_2)).
$

If the kernel \eqref{eq12.5} is positive definite, i.e. 
$
\sum\nolimits_{1\le i, j\le N} \bar c_i c_j k (\rho,y_i,y_j)\ge 0
$
for any points
$y_1,\dots,y_N\in \M$ and any complex numbers $c_1,\dots, c_N$,
then it can be thought of as a covariance of a Gaussian process (a random field)
on $\M$.
The standard methods of probability theory enable one 
to construct such random field  as a mapping  
$
W:\M\ni y\to W(y) = W(y,\omega )\in L_2(\Omega, d\omega),
$
such that $W(y_0)=~0,\bE W(y_1)=0$,
$\bE W(y_1)W(y_2)  =k(\rho,y_1,y_2)$ and
$\bE (W(y_1)-W(y_2))^2  =\rho (y_1,y_2)$,
for all $y_1,y_2\in \M$. Here $L_2(\Omega, d\omega)$ is the Hilbert space of real-valued
square-integrable random variables on a probability space $\Omega$ with a probability
measure $d\omega$ and $\bE$ denotes the expectation on $L_2(\Omega, d\omega)$.
Furthermore, if $\M$ is a Riemannian manifold and  
$\rho$ is H\"older continuous with respect to the geodesic distance $\theta$, i.e. 
$\rho (y_1,y_2) < c\theta (y_1,y_2)^{\beta}$ with some constants $c$ and $\beta >0$,
then for almost all $\omega \in \Omega$  trajectories of the random field
$W(y,\omega)$ are continuous functions of $y\in \M$. 
For more details we refer to \cite{19}.

Comparing the definitions \eqref{eq1.13} and \eqref{eq1.14} with \eqref{eq12.5},
we obtain
\begin{equation}
k(\theta^{\Delta}(\eta),y_1,y_2)= 
\int\limits_{\R} k(\theta^{\Delta}_r,y_1,y_2) \eta (r) \, dr,
\label{eq12.6}
\end{equation}
where
\begin{equation}
k(\theta^{\Delta}_r,y_1,y_2)= \int\limits_{\M} F_r (y_1,y) F_r (y_2,y) \, d\mu (y),
\label{eq12.7}
\end{equation}
where $F_r(x,y)=\chi (B_r(x),y)-\chi (B_r(y_0),y)$.
This proves that the L\'evi--Schoenberg kernels for all symmetric difference
metrics are positive definite. 
Particularly, in view of \eqref{eq4.14}, for the geodesic metric on
the sphere $S^d$ the formula \eqref{eq12.7} can be written as
\begin{equation}
k(\theta,y_1,y_2)= 
2\pi \int\limits_{S^d} F_{\pi/2} (y_1,y) F_{\pi/2} (y_2,y) \, d\mu (y).
\label{eq12.8}
\end{equation}  
Therefore, the kernel 
$\theta (y_1,y_0) +\theta (y_2,y_0) -\theta(y_1,y_2)$
is positive definite.
This is a well-known theorem of L\'evy, see \cite{25} and \cite{19}. 
Originally, its proof
was obtained in terms of 'white noise' integrals for random fields on $S^d$,
see \cite[Chap.~3 in Appendix]{25}. 
A direct proof was given in \cite[Sec.~4]{19}
in terms of an expansion of the metric $\theta$ by Gegenbauer polynomials.  
The proof of  L\'evy's theorem given above is likely to be the simplest.

Notice that in contrast to the spheres $S^d$, the geodesic metrics  $\theta$ 
on the
projective spaces $\C P^n$, $\H P^n$  and  $\Q P^2$ 
are not symmetric difference metrics and for projective spaces analogs of
L\'evy's theorem are not true. This follows from the results of the paper
\cite[Sec.~4, pp.~225--226]{19}.
At the same time, the L\'evi--Schoenberg kernel $k(\tau,y_1,y_2)$
for the chordal metric $\tau$
is positive definite for all two-point homogeneous spaces $Q(d,d_0)$.
This follows from Theorems \ref{thm2.1}.

A general theory of random fields on two-point homogeneous spaces  
has been developed in \cite{19}.
It should be interesting to study in more details random fields on 
$Q(d,d_0)$ with the covariances \eqref{eq12.6} and \eqref{eq12.7}. 
Notice that Lemma \ref{lem2.1} contains, in fact, conditions under which 
trajectories of such random field are continuous almost surely.


(iii) Finally, 
we notice that
noncompact connected two-point homogeneous spaces $G/K$ are also classified completely
as hyperbolic spaces over algebras $\F=\bR$, $\C$, $\H$, $\bO$, 
see ~\cite[Sec.~8.12]{36}, and one can consider the spaces of double cosets
$\M=\Gamma\setminus Q=\Gamma\setminus G/K$, where $\Gamma\subset G$ is a discrete 
subgroup in the group of isomerties of $Q$, such that the invariant 
measure $\mu(\M)<\infty$.
In this case, discrepancies of distributions and sums of pairwise distances 
for the symmetry difference metrics can be defined 
and their study should be of much interest, especially for non-compact $\M$ of finite measure. 


\end{document}